\documentclass[11pt]{amsart}
\usepackage[dvips]{graphicx}
\usepackage{epsfig}
\usepackage{amssymb}
\usepackage[usenames, dvipsnames]{color}
\usepackage{hyperref}
\usepackage{verbatim}
\usepackage[normalem]{ulem}
\newtheorem{theoremABC}{Theorem} 

\theoremstyle{plain}
\newtheorem{theorem}{Theorem}[section]
\newtheorem{lemma}[theorem]{Lemma}
\newtheorem{corollary}[theorem]{Corollary}

\theoremstyle{definition}

\theoremstyle{remark}
\newtheorem{remark}[theorem]{Remark}

\def\al{\alpha}

\def\O{\Omega}

\def\be{\begin{equation}}
\def\ee{\end{equation}}
\def\bes{\begin{equation*}}
\def\ees{\end{equation*}}
\def\bali{\begin{aligned}}
\def\eali{\end{aligned}}
\def\al{\begin{aligned}}
\def\eal{\end{aligned}}

\def\lab{\label}

\def\2O{\underline{\O}}

\numberwithin{equation}{section}

\makeatletter
\def\dashint{\operatorname%
{\,\,\text{\bf--}\kern-.98em\DOTSI\intop\ilimits@\!\!}}
\makeatother

\begin{document}

\title[backward Harnack]{Backward Harnack inequality and Hamilton estimates for heat type equations}

\author[Lu]{Juanling Lu}

    \address{School of Mathematical Sciences, East China Normal University, 500 Dongchuan Road, Shanghai 200241, P. R. of China,
    E-mail address:    lujuanlingmath@163.com,}

 \author[Wu]{ Yuting Wu}

 \address{College of Mathematics and Statistics, Northwest Normal University, Lanzhou 730070, Gansu, P. R. of China, E-mail address: 090094@nwnu.edu.cn,}

 \author[Zhang]{Qi S. Zhang}

  \address{Department of Mathematics, University of California, Riverside, CA 92521, USA, E-mail address: qizhang@math.ucr.edu.}

\date{}

\begin{abstract}

 Based on gradient estimates for the heat equation by Hamilton, we discover a  backward in time Harnack inequality for positive solutions on compact manifolds without further restrictions such as boundedness or vanishing boundary value for solutions. Contrary to the usual Harnack inequality, it allows comparison of the values of a solution at two different space time points, in both directions of time. In view of the importance of the usual Harnack inequality, further application is expected.

 Next we prove additional Li-Yau and Hamilton-type gradient estimates for positive solutions to the conjugate heat equation and the heat equation, coupled with the Ricci flow. In particular, the Li-Yau-type bound holds without any curvature conditions. Further more,  under different curvature assumptions from those in \cite{Zhang Qi and Li Xiaolong}, we prove some matrix Li-Yau-Hamilton estimates for positive solutions to the heat equation for type III Ricci flows.
\end{abstract}

\maketitle

\begin{center}
Dedicated to the memory of Professor R. S. Hamilton
\end{center}

\vspace{0.5cm}

\section{Introduction}

\subsection{Backward Harnack inequality}

In the influential paper \cite{Peter Li and Yau}, P. Li and S.-T. Yau showed how the classical Harnack principle for the heat equation on a manifold can be derived from a differential inequality. In particular, they discovered that if an n-dimensional compact Riemannian manifold $(\mathrm{\bf{M}}^{n},g)$ has a lower bound on its Ricci curvature, i.e., $\mathrm{Ric}\geq-Kg$, then all positive solutions $u: \mathrm{\bf{M}}^{n} \times[0,\infty) \rightarrow \mathbb{R}$ of the heat equation
\begin{equation}\label{def. of heat equation for fixed metric}
\partial_{t}u-\Delta_{g} u=0
\end{equation}
satisfy the following estimate
\begin{equation*}
\frac{|\nabla u|^{2}}{u^{2}}-\alpha\frac{\partial_{t}u}{u}
\leq\frac{n\alpha^{2}K}{2(\alpha-1)}+\frac{n\alpha^{2}}{2t},\quad \forall \alpha>1,t>0.
\end{equation*}
In the special case that $\mathrm{Ric}\geq0$, one obtains the optimal Li-Yau bound
$$
\frac{\partial_{t}u}{u}-\frac{|\nabla u|^{2}}{u^{2}}+\frac{n}{2t} = \Delta\ln u +\frac{n}{2t} \geq 0
$$
for all $(x,t) \in \mathrm{\bf{M}}\times(0,\infty)$. As mentioned, one of the many applications of this result is the Harnack inequality:
\begin{equation}\label{forward Harnack ineq.}
u(x,t_{1})\leq u(y,t_{2})\left(\frac{t_{2}}{t_{1}}\right)^{\frac{n\alpha}{2}}
\exp\left\{\frac{\alpha d^{2}(x,y)}{4(t_{2}-t_{1})}
+\frac{n\alpha K(t_{2}-t_{1})}{2(\alpha-1)}\right\}
\end{equation}
for any $\alpha>1$, $x,y\in\mathrm{\bf{M}}$, and $0<t_{1}<t_{2}$. It is well known that the Harnack inequality is an important and technically challenging result in the study of elliptic and parabolic equations.
The main idea is that the current temperature is bounded by future temperatures locally, with  a necessary time gap.
The first result of this paper is to prove that a reversed statement is also true on all compact manifolds, removing the required time gap in the classical parabolic Harnack inequality.
Backward Harnack inequalities have been found by a number of authors \cite{Garofalo}, \cite{Zhang and Han Qing} and \cite{Cheng Lijuan}. However, various restrictions on the solution such as boundedness or vanishing boundary values or nonlinear pattern  were imposed. In Theorem 3.1 \cite{Hamilton 1993}, a quasi-elliptic type Harnack inequality was also found under certain restrictions such as the integral of the solution at each time is $1$ and the distance is sufficiently small, together with volume non-collapsing condition on the manifold. There is also an error term $\delta$. A few years later, in the paper \cite{W}, an elliptic type Harnack inequality with power like nonlinearity on the solution was found, which  holds on both compact and some noncompact manifolds. See also \cite{ATW}. But when the power goes to one, the constant in the Harnack inequality goes to infinity.

Also, in \cite{Hamilton and Sesum}, Hamilton and Sesum proved an elliptic type Harnack inequality for positive solutions of the conjugate heat equation under the condition that they satisfy Perelman's differential Harnack inequality \cite[Corollary 9.3]{Perelman} and the K\"ahler Ricci flow is of type one with bounded diameter. It is known that not all positive solutions satisfy Perelman's Harnack inequality.

In this paper, we manage to remove all these restrictions for the heat equation on all compact manifolds. In Remark \ref{remark 1.3} below, we will explain that the backward Harnack inequality is optimal in some sense.

The proof of the backward Harnack inequality is actually based on heat kernel bounds in \cite{Peter Li and Yau} and the following gradient estimates for \eqref{def. of heat equation for fixed metric} in
another significant paper \cite{Hamilton 1993} by Hamilton.

\begin{theoremABC}(\cite[Theorem 1.1]{Hamilton 1993})\label{Theorem A}
Let $(\mathrm{\bf{M}}^{n}, g)$ be an n-dimensional compact manifold with $\mathrm{Ric} \geq -Kg$ for some constant $K>0$, and $u=u(x,t) \leq A$ be a positive solution of the heat equation (\ref{def. of heat equation for fixed metric}) on $\mathrm{\bf{M}}\times(0,\infty)$. Then
$$t|\nabla u|^{2} \leq (1+2Kt)u^{2}\ln\left(\frac{A}{u}\right)$$
holds on $\mathrm{\bf{M}}\times(0,\infty)$.
\end{theoremABC}

\begin{theoremABC}(\cite[Lemma 4.1]{Hamilton 1993})\label{Theorem B}
Let $(\mathrm{\bf{M}}^{n}, g)$ be an n-dimensional compact manifold with $\mathrm{Ric} \geq -Kg$ for some constant $K>0$, and $u=u(x,t)$ be a positive solution of the heat equation (\ref{def. of heat equation for fixed metric}) with $0<u\leq A$. Then there exists a constant $\theta_{1}$ depending only on $\mathrm{\bf{M}}$ such that
$$t\Delta u \leq \theta_{1}u\left[1+\ln\left(\frac{A}{u}\right)\right]$$
for $t\in[0,1]$.
\end{theoremABC}

Here is the statement of the theorem on the backward Harnack inequality.

\begin{theorem}\label{Thm-reverse Harnack inequality}
Let $(\mathrm{\bf{M}}^{n},g)$ be an n-dimensional compact Riemannian manifold without boundary and $u=u(x,t)$ be a positive solution of the heat equation (\ref{def. of heat equation for fixed metric}) in $\mathrm{\bf{M}}\times(0,1]$. Suppose the Ricci curvature is bounded from below by a non-positive constant $-K$, i.e., $\mathrm{Ric} \geq -K g$. Then for $\theta_{1}=e^{K}\max\{n,4\}$ and another positive constant $\theta_{2}$, depending only on $n$ and $K$, the inequality
$$\frac{u(x,t_2)}{u(z,t_1)} \leq \left(\frac{t_2}{t_1}\right)^{4\theta_{1}\theta_{2}}
\exp\left\{4\theta_{1}\theta_{2}\mathrm{diam^{2}}
\left(\frac{1}{t_1}-\frac{1}{t_2}\right)
+ 2 \sqrt{\theta_2} \left(\sqrt{K}+ \frac{1}{\sqrt{t_2}}\right)  \, \left(   1+ \frac{\mathrm{diam}}{\sqrt{t_2}} \right) \, d(x, z)\right\}$$
holds
for any $x,z \in \mathrm{\bf{M}}$ and $0<t_{1} \leq t_{2} \leq 1$, where $\mathrm{diam}$ is the diameter of $\mathrm{\bf{M}}$.
\end{theorem}

\begin{remark}
Particularly, if the manifold has nonnegative Ricci curvature, then the constant $\theta_{2}=1+n+8n^{2}+\ln\left(2^{3n}e^{4}(2\pi)^{\frac{n}{2}}\frac{n}{\alpha_{n-1}}\right)$ in Theorem \ref{Thm-reverse Harnack inequality}, where $\alpha_{n-1}$ is the area of the unit $(n-1)$ - sphere. See Appendix \ref{appendix B} for the detailed computation.
\end{remark}

\begin{remark}\label{remark 1.3}
Notice that the backward Harnack inequality still holds when $t_1=t_2$, allowing us to compare values of the solution at the same time. Also the inequality becomes an equality when $x=z$ and $t_2=t_1$.  In contrast, when $t_1=t_2$, the usual Harnack inequality \eqref{forward Harnack ineq.} becomes vacuous since the right-hand side becomes infinity unless $x=y$.
In addition, the appearance of $1/t_1$ on the right-hand side is necessary in view of the heat kernel. We mention that the Harnack inequality at the same time level on compact manifolds should be known to researchers in probability theory since the necessary spatial gradient bound for the heat kernel have been proven in \cite{Sheu}. See also \cite{Hsu} and \cite{Stroock}. However, these authors did not write out the dependence of their bounds on geometric invariance such as Ricci curvature, volume of geodesic balls or diameters. The explicit dependence was later given in \cite{Kots} and \cite{LiZz}.
The backward Harnack inequality (for $t_2>t_1$ case) seems new.

Naturally one wonders if the backward Harnack inequality is true for local solutions or on noncompact manifolds.
As explained in \cite{Garofalo} p544,  in general, the backward Harnack inequality does not hold for local solutions on compact domains. The backward Harnack inequality in current form with $\mathrm{diam}$ replaced by $d(x, z)$  is also not true on noncompact manifolds such as $\mathbb{R}^n$, as can be seen from the following explicit example.
By direct computation, the function
\[
u=(1-t)^{-1/2} e^{\frac{x^2}{4 (1-t)}}
\]is a positive solution to the heat equation in $\mathbb{R} \times [0, 1)$.
Note that the solution blows up at $t=1$. Also $u(x, t)/u(0, t)= e^{\frac{x^2}{4 (1-t)} } $, which is not dominated by any $e^{c x^2/t}$ as $t \to 1^-$. Here $c$ is a positive constant.

The backward Harnack inequality can also be regarded as a backward version of the Bakry-Qian estimate
\cite{Bakry and Qian Zhongmin} Theorem 9 which infers $\partial_t u /u \ge -C/t$,  as an application  of the Li-Yau gradient estimate.
A combination of the two allows us to compare values of positive solutions of the heat equation on compact manifolds at any two space time points.

\end{remark}

\subsection{Gradient estimates under the Ricci flow}

On a closed Riemannian manifold $\mathrm{\bf{M}}$, when the metrics $g(t)$ evolve by the Ricci flow (see \cite{Hamilton 1982} or \cite{Hamilton Ricci flow})
\begin{equation}\label{def. of Ricci flow}
\partial_{t}g=-2\mathrm{Ric},
\end{equation}
G. Perelman \cite{Perelman} obtained the following gradient estimate for the fundamental solution of the conjugate heat equation
\begin{equation}\label{def. of conjugate heat eq.}
\Delta_{g(t)}u+\partial_{t}u-R_{g(t)}u=0
\end{equation}
 without any curvature assumptions. Here $R=R_{g(t)}$ is the scalar curvature under the metric $g(t)$.

 Fixing $x_0 \in \mathbf{M}$,
let $u=G(x, \tau; x_0, 0)$ be the positive fundamental solution of (\ref{def. of conjugate heat eq.}) in $\mathrm{\bf{M}}\times[0,T]$ and $f$ be the function such that $u=(4\pi\tau)^{-\frac{n}{2}}e^{-f}$ with $\tau=T-t$. Then the following inequality holds
$$(\tau(2\Delta f-|\nabla f|^{2}+R)+f-n)u \leq 0,$$
or equivalently,
$$\frac{|\nabla u|^{2}}{u^{2}}-2\frac{\partial_{\tau}u}{u}-R \leq \frac{n}{\tau}+\frac{\ln u}{\tau}+\frac{n\ln(4\pi\tau)}{2\tau}.$$
This formula can be regarded as a generalization of the Li-Yau gradient estimate for the heat equation (\ref{def. of heat equation for fixed metric}). But it may not hold for any positive solutions. Subsequently, S. L. Kuang and Q. S. Zhang \cite{Kuang and Zhang} established a gradient estimate that works for all positive solutions of the conjugate heat equation (\ref{def. of conjugate heat eq.}) under the Ricci flow (\ref{def. of Ricci flow}) on a closed manifold, with some restrictions on the time interval. The second goal of the paper is to remove this restriction, without any curvature assumptions. Related results on the forward heat equation can be found in \cite{Zhang and Zhu Meng} and \cite{Song-Wu-Zhu}.

\begin{theorem}\label{Kuang-Zhang Thm1-1}
Let $(\mathrm{\bf{M}}^{n},g(t)), t\in[0,T]$, be a smooth solution to the Ricci flow (\ref{def. of Ricci flow}) on an n-dimensional closed Riemannian manifold, and $u: \mathrm{\bf{M}}\times[0,T)\rightarrow(0,\infty)$ be a positive $C^{2,1}$ solution to the conjugate heat equation (\ref{def. of conjugate heat eq.}). Assume $f$ is the function such that $u=(4\pi\tau)^{-\frac{n}{2}}e^{-f}$ with $\tau=T-t$. Then the following inequality
\begin{equation*}
2\Delta f-|\nabla f|^{2}+R\leq \frac{2n}{t}+\frac{2n}{\tau}
\end{equation*}
holds for all $t\in(0,T)$ and all points.
\end{theorem}

\begin{remark}
In Theorem \ref{Kuang-Zhang Thm1-1}, the appearance of $\frac{1}{t}$ on the right-hand side is necessary since no assumption is made on the initial manifold whose curvature can be arbitrary. The above inequality is equivalent to
\[
\frac{|\nabla u|^{2}}{u^{2}}-2\frac{\partial_{\tau}u}{u}-R\leq\frac{2n}{t}+\frac{2n}{\tau}.
\]

\end{remark}

Observe that for each $t$, at the infimum point of $f(\cdot, t)$, we have $\nabla f =0$ and $\Delta f \ge 0$. Therefore, the theorem yields the following application.

\begin{corollary}
Let $(\mathrm{\bf{M}}^{n},g(t)), t\in[0,T]$, be a smooth solution to the Ricci flow (\ref{def. of Ricci flow}) on an n-dimensional closed Riemannian manifold, then the scalar curvature $R$ satisfies
\[
\inf \, R(\cdot, t) \le \frac{2n}{t}+\frac{2n}{T-t}.
\]

\end{corollary}

\subsection{Matrix estimates under the Ricci flow}
Now we turn to matrix gradient estimates for heat-type equations.
In \cite{Hamilton 1993}, R. S. Hamilton not only proved Theorems \ref{Theorem A} and \ref{Theorem B}, but also applied Theorems \ref{Theorem A} and \ref{Theorem B} to prove the following matrix Harnack estimate.

\begin{theoremABC}(\cite[Corollary 4.4]{Hamilton 1993})\label{Theorem C}
Let $(\mathrm{\bf{M}}^{n}, g)$ be a compact n-dimensional manifold with non-negative sectional curvature and parallel Ricci curvature. Then the matrix estimate $$\nabla_{i}\nabla_{j}\ln u+\frac{1}{2t}g_{ij} \geq 0$$ is satisfied by any positive solution $u$ of the heat equation (\ref{def. of heat equation for fixed metric}).
\end{theoremABC}
It is noted that this result generalizes the Li-Yau estimate to a full matrix version under stronger curvature assumptions. Later on, B. Chow and R. S. Hamilton \cite{Bennett and Hamilton-Harnack} further extended the Li-Yau estimate and Theorem \ref{Theorem C} to the constrained case under the same curvature assumptions, and discovered new linear Harnack estimates. Analogous matrix estimate was obtained by H.-D. Cao and L. Ni \cite{Cao Huaidong and Ni Lei-Kahler} on K\"{a}hler manifolds with non-negative bisectional curvature, which is called a matrix Li-Yau-Hamilton estimate. Additionally, Q. Han and one of us \cite{Zhang and Han Qing} investigated global and local upper bounds for the Hessian of log positive solutions of the heat equation on a Riemannian manifold.

When $(\mathrm{\bf{M}}^{n}, g(t))$ is a complete solution to the K\"{a}hler-Ricci flow with bounded non-negative bisectional curvature, L. Ni \cite{Ni Lei-kahler} proved that a positive solution $u$ of the forward conjugate heat equation
\begin{equation}\label{forward conjugate eq}
(\Delta_{g(t)}+R_{g(t)}-\partial_{t})u=0
\end{equation}
satisfies the estimate
$$
R_{\alpha\bar{\beta}}+\nabla_{\alpha}\nabla_{\bar{\beta}}\ln u+\frac{1}{t}g_{\alpha\bar{\beta}} \geq 0.
$$
The equality holds if and only if $(\mathrm{\bf{M}}^{n}, g(t))$ is an expanding K\"{a}hler-Ricci soliton. These results were generalized to the constrained case by X.-A. Ren, S. Yao, L.-J. Shen, and G.-Y. Zhang \cite{Ren Xin-an 2015}.

Recently, X. L. Li and one of us \cite{Zhang Qi and Li Xiaolong} proved new matrix Li-Yau-Hamilton estimates for positive solutions to the (forward) heat equation
\begin{equation}\label{def. of heat eq.}
\partial_{t}u-\Delta_{g(t)}u=0
\end{equation}
and the conjugate heat equation (\ref{def. of conjugate heat eq.}) coupled with the Ricci flow (\ref{def. of Ricci flow}). Specifically, if $(\mathrm{\bf{M}}^{n}, g(t))$ has non-negative sectional curvature and $\mathrm{Ric} \leq \kappa g$ for some constant $\kappa > 0$, then positive solutions $u$ of the heat equation (\ref{def. of heat eq.}) obey the following estimate
$$\nabla_{i}\nabla_{j}\ln u+\frac{\kappa}{1-e^{-2\kappa t}}g_{ij} \geq 0.$$
Similarly, all positive solutions $u$ of the conjugate heat equation (\ref{def. of conjugate heat eq.}) satisfy
$$R_{ij}-\nabla_{i}\nabla_{j}\ln u-\eta g_{ij} \leq 0$$
provided that $(\mathrm{\bf{M}}^{n}, g(t))$ has non-negative complex sectional curvature and $\mathrm{Ric} \leq \kappa g$ for some $\kappa > 0$, where $\eta$ is a time-dependent function satisfying the inequality $\eta' \leq 2\eta^{2}-2\kappa\eta-\frac{\kappa}{t}$. They also used these estimates to study the monotonicity of various parabolic frequencies. In \cite{Li Xiaolong and Ren Xin-An 2025}, X. L. Li, H.-Y. Liu, and X.-A. Ren proved analogous results under the K\"{a}hler-Ricci flow. Subsequently, S. Yao, H.-Y. Liu, and X.-A. Ren \cite{Yao and Liu and Ren} derived the matrix Li-Yau-Hamilton estimate for positive solutions to some nonlinear heat equations under the Ricci flow (\ref{def. of Ricci flow}).

Although it would be desirable to remove the upper bound condition on the Ricci curvature, we can only replace it by the type III curvature condition so far.

\begin{theorem}\label{matric heat eq.-Thm}

Let $(\mathrm{\bf{M}}^n, g(t))$, $t \in [0,T]$, $T>0$, be a solution to the Ricci flow (\ref{def. of Ricci flow}) on an n-dimensional compact Riemannian manifold. Assume that $u: \mathrm{\bf{M}} \times [0,T] \to \mathbb{R}$ is a positive solution to the heat equation (\ref{def. of heat eq.}).

Suppose that $(\mathrm{\bf{M}}^n, g(t))$ has non-negative sectional curvature  and  the Ricci curvature satisfies the type III condition $|\mathrm{Ric} (\cdot, t)| \le \frac{\sigma}{t}$ for some constant $\sigma$,
then
$$\nabla_i \nabla_j \ln u + \frac{\beta}{t} g_{ij} \geq 0$$
for all $(x,t) \in \mathrm{\bf{M}} \times (0,T)$ and any constant $\beta \geq \frac{1+2\sigma}{2}$.

\end{theorem}

\begin{remark}  The following are some concrete cases where the theorem is applicable.
In dimension $n=3$, suppose that $(\mathrm{\bf{M}}^3, g(t))$ has non-negative sectional curvature. Then there exists a small time $T>0$, depending only on the volume non-collapsing
constant $v_0 \equiv \inf_{x \in  \textbf{M}}|B(x, 1)|_{g(0)}$, such that
$$\nabla_i \nabla_j \ln u + \left(\frac{C_0}{t}\right) g_{ij} \geq 0$$
for any $t \in (0,T)$ and  a positive  constant $C_{0}$  depending only on $v_0$.
This is consequence of the small time type III curvature bound from \cite[Lemma 2.1]{topping} and the theorem.

 For all $n \ge 4$, assume the isoperimetric constant $i_n$ of the initial manifold is sufficiently close to the Euclidean ones, then there exists a small time $T>0$, depending only on $i_n $, such that
$$\nabla_i \nabla_j \ln u + \left(\frac{C_0}{t}\right) g_{ij} \geq 0$$
for any $t \in (0,T)$ and  a positive  constant $C_{0}$  depending only on $i_n$.
This follows from the small time type III curvature bound from  Perelman's Pseudolocality  Theorem 10.1 in \cite{Perelman} and  the theorem.
\end{remark}

The rest of the paper is organized as follows. In Section \ref{section 2}, we will prove the backward Harnack inequality Theorem \ref{Thm-reverse Harnack inequality}. In Section \ref{section 3} we will prove Theorem \ref{Kuang-Zhang Thm1-1} and introduce some related gradient estimates Theorem \ref{Hamilton Thm of grad(u)} and \ref{Hamilton Thm of Laplace(u)}.
Finally Theorem \ref{matric heat eq.-Thm} will be proven in Section \ref{section 4}. Unless stated otherwise, we use $C$, $c$ with or without index to denote positive constants that may change from line to line. Also $B(x, r)$  denotes the geodesic ball centered at $x$ with radius $r$, $|B(x, r)|$ the volume and $d(x, y)$ the distance between two points $x, y$ on $\mathbf{M}$, a Riemannian manifold.

\section{Proof of Theorem \ref{Thm-reverse Harnack inequality}}\label{section 2}

In this section, we apply Theorems \ref{Theorem A} and \ref{Theorem B} to give a proof of the backward Harnack inequality.

\begin{proof}[\bf{Proof of Theorem \ref{Thm-reverse Harnack inequality}}]
The proof is divided into three steps.

\textbf{Step 1.} We derive the pertinent inequality for the heat kernel with fixed spatial point at different times.

In this step, we assume $G=G(x,t;y,0)$ is the heat kernel with a pole at $y \in \mathrm{\bf{M}}$, $t=0$. Given any fixed $T \in (0,1]$, by \cite[Lemma 4.1]{Hamilton 1993} (Theorem \ref{Theorem B} here), we have, for
\[
A=\sup_{\mathrm{\bf{M}}\times[\frac{T}{2},T]}G(x,t;y,0),
\] that
\begin{equation}\label{Hamilton-laplace bound}
\frac{\partial_{t}G}{G} \leq \frac{\theta_{1}}{t-\frac{T}{2}}\ln\left(\frac{eA}{G}\right)
\end{equation}
for $t\in [\frac{T}{2},T]$, where $\theta_{1}=e^{K}\max\{n,4\}$ (See Appendix \ref{appendix A} for details).
It follows from \cite[p267]{Zhang-CAG}, fourth line from below, which holds for $t\leq 1$, that
\begin{equation}\label{Zhang-CAG-log(A/u) bound}
\ln\left( \frac{e A}{G}\right) \leq \theta_{2}\left(1+\frac{\mathrm{diam}^{2}}{t}\right)
\end{equation}
for $t\in [\frac{T}{2},T]$, where $\theta_{2}$ is a constant depending only on $n$ and $K$.

Now we choose $t\in[\frac{3}{4}T,T]$.
Combining (\ref{Zhang-CAG-log(A/u) bound}) with (\ref{Hamilton-laplace bound}) and $t-\frac{T}{2} \geq \frac{1}{4}T \geq \frac{1}{4}t$, we conclude
\begin{equation}\label{boound of dG/dt}
\partial_{t}\ln G(x,t;y,0) \leq \frac{4\theta_{1} \theta_{2}}{t}+\frac{4\theta_{1}\theta_{2}\mathrm{diam}^{2}}{t^{2}}
\end{equation}
for any $x\in \mathrm{\bf{M}}$ and $t\in [\frac{3}{4}T,T]$. Since $T\in(0,1]$ is arbitrary, we deduce that (\ref{boound of dG/dt}) holds for all $t\in(0,1]$.
Integrating inequality (\ref{boound of dG/dt}) from $t_{1}$ to $t_{2}$, we obtain
$$\ln\left(\frac{G(x,t_{2};y,0)}{G(x,t_{1};y,0)}\right) \leq 4\theta_{1} \theta_{2} \ln\left(\frac{t_{2}}{t_{1}}\right)+4\theta_{1}\theta_{2}\mathrm{diam}^{2}\left(\frac{1}{t_{1}}-\frac{1}{t_{2}}\right),
\quad t_{1},t_{2}\in(0,1].$$
Exponentiating both sides of the inequality yields
\begin{equation}\label{heat kernel bound-fixed metric}
G(x,t_{2};y,0) \leq G(x,t_{1};y,0)\left(\frac{t_{2}}{t_{1}}\right)^{4\theta_{1} \theta_{2}}
\exp\left\{4\theta_{1}\theta_{2}\mathrm{diam}^{2}\left(\frac{1}{t_{1}}-\frac{1}{t_{2}}\right)\right\}
\end{equation}
for any $x\in\mathrm{\bf{M}}$ and $0<t_{1} \leq t_{2}\leq 1$.

\textbf{Step 2.} Notice that the first spatial point in $G$ in the above inequality is fixed.
In this step, we will prove that the heat kernel $G$ at different spatial points in the same time can be compared too, modulo an exponential term involving the distance function. This is based on Theorem \ref{Theorem A} by Hamilton.

Let $u=G(x, t; y, 0)$ and $A=\sup_{(x, t) \in \mathbf{M }\times [T/2, T]} u$. Then the aforementioned theorem (Theorem \ref{Theorem A}) infers, for $t \in (T/2, T]$, that
$$(t-(T/2))|\nabla \ln (A/u)|^{2} \leq (1+2K (t-(T/2)) \, \ln\left(\frac{A}{u}\right).$$
Denoting $f=f(x, t) := \ln (A/u(x, t))$, we can write the above inequality as
$$\frac{|\nabla f|^{2}}{f} \leq \frac{(1+2K (t-(T/2))}{(t-(T/2))}=2K+ \frac{1}{(t-(T/2))}.$$
Picking two points $x_1, x_2$ in $\mathbf{M}$ and taking square root of the preceding inequality, we deduce, after integrating along a minimal geodesic, that
\[
\sqrt{\ln \left( A/u(x_1, t)\right)} - \sqrt{\ln \left( A/u(x_2, t)\right)} \le \frac{1}{2} \left[2K+ \frac{1}{(t-(T/2))}\right]^{1/2} \, d(x_2, x_1).
\]Therefore
\[
\ln \left( u(x_2, t)/u(x_1, t)\right)  \le \frac{1}{2} \left[2K+ \frac{1}{(t-(T/2))}\right]^{1/2} \, d(x_2, x_1) \, \left( \sqrt{\ln \left( A/u(x_1, t)\right)} + \sqrt{\ln \left( A/u(x_2, t)\right)} \right),
\]which implies, by \eqref{Zhang-CAG-log(A/u) bound}, that
\[
\ln \left( u(x_2, t)/u(x_1, t)\right)  \le \frac{1}{2} \left[2K+ \frac{1}{(t-(T/2))}\right]^{1/2} \, d(x_2, x_1) \, \left( 2 \sqrt{\theta_2} (1+ \mathrm{diam}/\sqrt{t}) \right).
\]For $t \in [\frac{3}{4}T, T]$, this infers
\[
\ln \left( u(x_2, t)/u(x_1, t)\right)  \le 2 \left[\sqrt{K}+ \frac{1}{\sqrt{t}}\right] \, d(x_2, x_1) \, \left(  \sqrt{\theta_2} (1+ \mathrm{diam}/\sqrt{t}) \right),
\]namely
\be
\lab{gx2x1t}
G(x_2, t; y, 0) \le G(x_1, t; y, 0) \, \exp \left[2 \sqrt{\theta_2} \left(\sqrt{K}+ \frac{1}{\sqrt{t}}\right)  \, \left(   1+ \frac{\mathrm{diam}}{\sqrt{t}} \right) \, d(x_2, x_1) \right].
\ee Let us mention that this bound with $\mathrm{diam}$ replaced by $d(x_1, x_2)$  can also be obtained from integrating the inequality
\[
|\nabla \ln G(x, t; y, 0)| \le C [d(x, y)/t + 1/\sqrt{t}],  \quad t \in [0, 1],
\]which was proven in \cite{Sheu}. However the constant $C$ was not explicitly given and could depend on the diameter and volume non-collapsing constant and the Ricci curvature lower bound. Later the constant was proven to depend only on $n$ and $K$ in \cite{Kots} and \cite{LiZz}.
By \eqref{gx2x1t} and \eqref{heat kernel bound-fixed metric}, this implies
\be
\lab{gx21t21}
\begin{aligned}
G(x_2,t_{2};y,0) &\leq G(x_1,t_{1};y,0)\left(\frac{t_{2}}{t_{1}}\right)^{4\theta_{1} \theta_{2}}
\exp\left\{4\theta_{1}\theta_{2}\mathrm{diam}^{2}\left(\frac{1}{t_{1}}-\frac{1}{t_{2}}\right)\right\}
\, \\
 &\qquad \times \exp \left[2 \sqrt{\theta_2} \left(\sqrt{K}+ \frac{1}{\sqrt{t_2}}\right)  \, \left(   1+ \frac{\mathrm{diam}}{\sqrt{t_2}} \right) \, d(x_2, x_1) \right].
\end{aligned}
\ee

\textbf{Step 3.} We prove that any positive solution to the heat equation \eqref{def. of heat equation for fixed metric} satisfies the same bound as \eqref{gx21t21}.

Let $u$ be any positive solution of the heat equation \eqref{def. of heat equation for fixed metric} on $\mathrm{\bf{M}}\times (0,1]$. Without loss of generality, we can assume $u_{0} = u(x, 0)$ is smooth because we can consider $u(x, t + \epsilon)$ and let $\epsilon \rightarrow0$ otherwise. Then
$$u(x,t)=\int_{\mathrm{\bf{M}}} G(x,t;y,0)u_{0}(y)dy$$
for any $(x,t)\in \mathrm{\bf{M}}\times(0,1]$.
Combining this with \eqref{gx21t21}, we find, for $z \in \mathbf{M}$, that

\[
\al
u(x,t_{2}) &\leq \left(\frac{t_{2}}{t_{1}}\right)^{4\theta_{1}\theta_{2}}
\exp\left\{4\theta_{1}\theta_{2}\mathrm{diam}^{2}\left(\frac{1}{t_{1}}-\frac{1}{t_{2}}\right) + 2 \sqrt{\theta_2} \left(\sqrt{K}+ \frac{1}{\sqrt{t_2}}\right)  \, \left(   1+ \frac{\mathrm{diam}}{\sqrt{t_2}} \right) \, d(x, z) \right\}\\
& \qquad \times
\int_{\mathrm{\bf{M}}} G(z,t_{1};y,0)u_{0}(y) dy.
\eal
\]
Hence, for any $x, z \in\mathrm{\bf{M}}$ and $0<t_{1} \leq t_{2}\leq 1$, the following holds
$$\frac{u(x,t_2)}{u(z,t_1)} \leq \left(\frac{t_2}{t_1}\right)^{4\theta_{1} \theta_{2}}
\exp\left\{4\theta_{1}\theta_{2}\mathrm{diam^{2}}
\left(\frac{1}{t_1}-\frac{1}{t_2}\right)
+ 2 \sqrt{\theta_2} \left(\sqrt{K}+ \frac{1}{\sqrt{t_2}}\right)  \, \left(   1+ \frac{\mathrm{diam}}{\sqrt{t_2}} \right) \, d(x, z)\right\}.$$

\end{proof}

\section{Gradient estimates for positive solutions to the heat-type equation under the Ricci flow}\label{section 3}

\subsection{Proof of Theorem \ref{Kuang-Zhang Thm1-1}}

Here, we present the proof of Theorem \ref{Kuang-Zhang Thm1-1}, which, as mentioned, is a Li-Yau type estimate for the conjugate heat equation free of curvature condition.

\begin{proof}
By a standard approximation argument as in \cite[vol.2]{Bennett Chow 2007} e.g., we can assume without loss of generality that $g=g(t)$ is smooth in the closed time interval $[0, T]$ and that $u$ is strictly positive everywhere. Indeed, by Theorem A.23 in \cite[vol.2]{Bennett Chow 2007} (due to W.X. Shi), the curvature tensor is uniformly bounded in the time interval $[0, T-\omega]$ with the bound depending only on the initial data and $\omega$, a positive number. Moreover the lower bound of the scalar curvature is nondecreasing about $t$ since the scalar curvature $R$ satisfies (c.f.\cite[p209]{Bennett Chow 2004})
$$\Delta R-\partial_{t}R+\frac{2}{n}R^{2}\leq0.$$
Therefore, we can just work on the interval $[0, T-\omega]$ first. In the proof, it will be clear that all constants are independent of the curvature tensor. Hence we can take $\omega$ to zero to get the desired result on $[0, T)$.

As done in \cite{Kuang and Zhang}, by standard computation, we have
\begin{equation*}
(\partial_{t}+\Delta)(\Delta f)=2\langle\text{Ric}, \nabla^{2}f\rangle+\Delta|\nabla f|^{2}-\Delta R.
\end{equation*}
Using the evolution equation of $g$ and $\partial_{t}f=-\Delta f+|\nabla f|^{2}-R+\dfrac{n}{2\tau}$, we deduce
\begin{equation*}
(\partial_{t}+\Delta)(|\nabla f|^{2})=2\text{Ric}(\nabla f, \nabla f)+2\langle\nabla f, \nabla(-\Delta f+|\nabla f|^{2}-R)\rangle
+\Delta|\nabla f|^{2}.
\end{equation*}
Notice also
\begin{equation*}
(\partial_{t}+\Delta)R=2\Delta R+2|\text{Ric}|^{2}.
\end{equation*}
Combining the above three expressions, we conclude
\begin{equation}\label{evolution equation}
\begin{aligned}
&(\partial_{t}+\Delta)(2\Delta f-|\nabla f|^{2}+R)\\
=&4\langle\text{Ric}, \nabla^{2}f\rangle+\Delta|\nabla f|^{2}
-2\text{Ric}(\nabla f, \nabla f)-2\langle\nabla f, \nabla(-\Delta f+|\nabla f|^{2}-R)\rangle+2|\text{Ric}|^{2}.
\end{aligned}
\end{equation}
Denote
\begin{equation*}
q(x,t):=2\Delta f-|\nabla f|^{2}+R.
\end{equation*}
By Bochner's identity
\begin{equation*}
\Delta|\nabla f|^{2}=2|\nabla^{2}f|^{2}+2\langle\nabla f, \nabla(\Delta f)\rangle+2\text{Ric}(\nabla f, \nabla f),
\end{equation*}
the equation (\ref{evolution equation}) becomes
\begin{equation*}
(\partial_{t}+\Delta)q-2\langle\nabla f, \nabla q\rangle\geq\frac{2}{n}(R+\Delta f)^{2},
\end{equation*}
this inequality can also be found independently in reference \cite[p202: the proof of Lemma 5.19]{Bennett Chow 2007}.

Since
\begin{equation*}
R+\Delta f=\frac{1}{2}(q+|\nabla f|^{2}+R),
\end{equation*}
and hence
\begin{equation}\label{the evolution eq of q}
\begin{aligned}
(\partial_{t}+\Delta)q-2\langle\nabla f, \nabla q\rangle\geq\frac{1}{2n}(q+|\nabla f|^{2}+R)^{2}.
\end{aligned}
\end{equation}
Let $\kappa\geq2n$ be a constant to be determined later. For any $\epsilon>0$, we get
\begin{equation}\label{the evolution eq of c/t}
\begin{aligned}
(\partial_{t}+\Delta)\left(\frac{\kappa}{t+\epsilon}+\frac{\kappa}{\tau+\epsilon}\right)
-2\left\langle\nabla f, \nabla\left(\frac{\kappa}{t+\epsilon}+\frac{\kappa}{\tau+\epsilon}\right)\right\rangle
=-\frac{\kappa}{(t+\epsilon)^{2}}+\frac{\kappa}{(\tau+\epsilon)^{2}}.
\end{aligned}
\end{equation}
Denote
\begin{equation*}
B:=|\nabla f|^{2}+R.
\end{equation*}
Subtracting inequality (\ref{the evolution eq of q}) by equation (\ref{the evolution eq of c/t}) yields
\begin{equation}\label{55555}
\begin{aligned}
&(\partial_{t}+\Delta)\left[q-\left(\frac{\kappa}{t+\epsilon}+\frac{\kappa}{\tau+\epsilon}\right)\right]
-2\left\langle\nabla f, \nabla\left[q-\left(\frac{\kappa}{t+\epsilon}+\frac{\kappa}{\tau+\epsilon}\right)\right]\right\rangle\\
\geq&\frac{1}{2n}(q+B)^{2}+\frac{\kappa}{(t+\epsilon)^{2}}-\frac{\kappa}{(\tau+\epsilon)^{2}}.
\end{aligned}
\end{equation}
Combining \eqref{55555} with the expression
\begin{equation*}
\begin{aligned}
&\frac{1}{2n}(q+B)^{2}+\frac{\kappa}{(t+\epsilon)^{2}}-\frac{\kappa}{(\tau+\epsilon)^{2}}\\
=&\frac{1}{2n}\left[(q+B)^{2}-\left(\frac{\kappa}{t+\epsilon}+\frac{\kappa}{\tau+\epsilon}\right)^{2}
+\left(\frac{\kappa}{t+\epsilon}+\frac{\kappa}{\tau+\epsilon}\right)^{2}
+\frac{2n\kappa}{(t+\epsilon)^{2}}-\frac{2n\kappa}{(\tau+\epsilon)^{2}}\right],
\end{aligned}
\end{equation*}
we obtain
\begin{equation}\label{the evolution eq of q+c/t}
\begin{aligned}
&(\partial_{t}+\Delta)\left[q-\left(\frac{\kappa}{t+\epsilon}+\frac{\kappa}{\tau+\epsilon}\right)\right]
-2\left\langle\nabla f, \nabla\left[q-\left(\frac{\kappa}{t+\epsilon}+\frac{\kappa}{\tau+\epsilon}\right)\right]\right\rangle\\
\geq&\frac{1}{2n}\left\{\left[q+B+\left(\frac{\kappa}{t+\epsilon}+\frac{\kappa}{\tau+\epsilon}\right)\right]
\left[q+B-\left(\frac{\kappa}{t+\epsilon}+\frac{\kappa}{\tau+\epsilon}\right)\right]
+\frac{\kappa^{2}+2n\kappa}{(t+\epsilon)^{2}}+\frac{\kappa^{2}-2n\kappa}{(\tau+\epsilon)^{2}}\right.\\
&\left.+\frac{2\kappa^{2}}{(t+\epsilon)(\tau+\epsilon)}\right\}.
\end{aligned}
\end{equation}
We hope to show that the right-hand side of inequality (\ref{the evolution eq of q+c/t}) is non-negative, thus we deal with the previous inequality at a given point $(x,t)$ in three cases.

\noindent{\bf{Case 1.}} $B\geq0$ and $q+B+\left(\frac{\kappa}{t+\epsilon}+\frac{\kappa}{\tau+\epsilon}\right)\leq0$. Then
\begin{equation*}
q+B-\left(\frac{\kappa}{t+\epsilon}+\frac{\kappa}{\tau+\epsilon}\right)\leq0,
\end{equation*}
also since $\kappa\geq2n$, we have
$$\frac{\kappa^{2}+2n\kappa}{(t+\epsilon)^{2}}
+\frac{\kappa^{2}-2n\kappa}{(\tau+\epsilon)^{2}}+\frac{2\kappa^{2}}{(t+\epsilon)(\tau+\epsilon)}\geq0.$$
Therefore
\begin{equation*}
(\partial_{t}+\Delta)\left[q-\left(\frac{\kappa}{t+\epsilon}+\frac{\kappa}{\tau+\epsilon}\right)\right]
-2\left\langle\nabla f, \nabla\left[q-\left(\frac{\kappa}{t+\epsilon}+\frac{\kappa}{\tau+\epsilon}\right)\right]\right\rangle\geq0.
\end{equation*}

\noindent{\bf{Case 2.}} $B\geq0$ and $q+B+\left(\frac{\kappa}{t+\epsilon}+\frac{\kappa}{\tau+\epsilon}\right)>0$. Then the inequality (\ref{the evolution eq of q+c/t}) can be changed to
\begin{equation*}
\begin{aligned}
&(\partial_{t}+\Delta)\left[q-\left(\frac{\kappa}{t+\epsilon}+\frac{\kappa}{\tau+\epsilon}\right)\right]
-2\left\langle\nabla f, \nabla\left[q-\left(\frac{\kappa}{t+\epsilon}+\frac{\kappa}{\tau+\epsilon}\right)\right]\right\rangle\\
&-\frac{1}{2n}\left[q+B+\left(\frac{\kappa}{t+\epsilon}+\frac{\kappa}{\tau+\epsilon}\right)\right]
\left[q-\left(\frac{\kappa}{t+\epsilon}+\frac{\kappa}{\tau+\epsilon}\right)\right]\\
\geq&\frac{1}{2n}B\left[q+B+\left(\frac{\kappa}{t+\epsilon}+\frac{\kappa}{\tau+\epsilon}\right)\right]\\
\geq&0.
\end{aligned}
\end{equation*}

\noindent{\bf{Case 3.}} $B<0$. Then the inequality (\ref{the evolution eq of q+c/t}) can be transformed into
\begin{equation}\label{the evolution eq of case 3}
\begin{aligned}
&(\partial_{t}+\Delta)\left[q-\left(\frac{\kappa}{t+\epsilon}+\frac{\kappa}{\tau+\epsilon}\right)\right]
-2\left\langle\nabla f, \nabla\left[q-\left(\frac{\kappa}{t+\epsilon}+\frac{\kappa}{\tau+\epsilon}\right)\right]\right\rangle\\
\geq&\frac{1}{2n}\left\{\left[q+B+\left(\frac{\kappa}{t+\epsilon}+\frac{\kappa}{\tau+\epsilon}\right)\right]
\left[q-\left(\frac{\kappa}{t+\epsilon}+\frac{\kappa}{\tau+\epsilon}\right)\right]
+B\left[q-\left(\frac{\kappa}{t+\epsilon}+\frac{\kappa}{\tau+\epsilon}\right)\right]\right.\\
&\left.+B^{2}+2B\left(\frac{\kappa}{t+\epsilon}+\frac{\kappa}{\tau+\epsilon}\right)+\frac{\kappa^{2}+2n\kappa}{(t+\epsilon)^{2}}
+\frac{\kappa^{2}-2n\kappa}{(\tau+\epsilon)^{2}}+\frac{2\kappa^{2}}{(t+\epsilon)(\tau+\epsilon)}\right\}.
\end{aligned}
\end{equation}
Using the weak minimum principle for a differential inequality $(\partial_{t}-\Delta)R\geq\frac{2}{n}R^{2}$, we have the following well known estimate of the scalar curvature $R$ under the Ricci flow (\ref{def. of Ricci flow}):
\begin{equation*}
R\geq-\frac{n}{2(t+\epsilon)}
\end{equation*}
for some $\epsilon>0$ depending on the initial value of $R$. Hence
\begin{equation*}
B=|\nabla f|^{2}+R\geq R\geq-\frac{n}{2(t+\epsilon)}
\end{equation*}
for any $t\in(0, T)$ and
\begin{equation*}
2B\left(\frac{\kappa}{t+\epsilon}+\frac{\kappa}{\tau+\epsilon}\right)
\geq-\frac{n}{t+\epsilon}\left(\frac{\kappa}{t+\epsilon}+\frac{\kappa}{\tau+\epsilon}\right).
\end{equation*}
Hence, combining the above inequality, we can transform expression (\ref{the evolution eq of case 3}) into
\begin{equation*}
\begin{aligned}
&(\partial_{t}+\Delta)\left[q-\left(\frac{\kappa}{t+\epsilon}+\frac{\kappa}{\tau+\epsilon}\right)\right]
-2\left\langle\nabla f, \nabla\left[q-\left(\frac{\kappa}{t+\epsilon}+\frac{\kappa}{\tau+\epsilon}\right)\right]\right\rangle\\
&-\frac{1}{2n}\left[q+2B+\left(\frac{\kappa}{t+\epsilon}+\frac{\kappa}{\tau+\epsilon}\right)\right]
\left[q-\left(\frac{\kappa}{t+\epsilon}+\frac{\kappa}{\tau+\epsilon}\right)\right]\\
\geq&\frac{1}{2n}\left[\frac{\kappa^{2}+n\kappa}{(t+\epsilon)^{2}}
+\frac{\kappa^{2}-2n\kappa}{(\tau+\epsilon)^{2}}+\frac{2\kappa^{2}-n\kappa}{(t+\epsilon)(\tau+\epsilon)}\right].
\end{aligned}
\end{equation*}
Taking $\kappa=2n$, we have
\begin{equation*}
\begin{aligned}
&(\partial_{t}+\Delta)\left[q-\left(\frac{2n}{t+\epsilon}+\frac{2n}{\tau+\epsilon}\right)\right]
-2\left\langle\nabla f, \nabla\left[q-\left(\frac{2n}{t+\epsilon}+\frac{2n}{\tau+\epsilon}\right)\right]\right\rangle\\
&-V\left[q-\left(\frac{2n}{t+\epsilon}+\frac{2n}{\tau+\epsilon}\right)\right]\\
\geq&0,
\end{aligned}
\end{equation*}
where $V=V(x,t)$ is a bounded function defined by
\[
V=
\begin{cases}
  \nonumber 0, & \, \text{if} \,  B\geq0, q+B+\left(\frac{2n}{t+\epsilon}+\frac{2n}{\tau+\epsilon}\right)\leq0, \,  \text{at} \, (x, t),\\
  \nonumber \frac{1}{2n}\left[q+B+\left(\frac{2n}{t+\epsilon}+\frac{2n}{\tau+\epsilon}\right)\right],
             &\, \text{if} \,  B\geq0, q+B+\left(\frac{2n}{t+\epsilon}+\frac{2n}{\tau+\epsilon}\right)>0, \, \text{ at} \, (x, t),\\
  \nonumber \frac{1}{2n}\left[q+2B+\left(\frac{2n}{t+\epsilon}+\frac{2n}{\tau+\epsilon}\right)\right], &\, \text{if} \,  B<0 , \text{at} \, (x, t).
\end{cases}
\]

Since we are assuming that the Ricci flow (\ref{def. of Ricci flow}) is smooth in $[0, T]$ and that $u(x,t)$ is a positive $C^{2,1}$ solution to the conjugate heat equation (\ref{def. of conjugate heat eq.}), we obtain
\begin{equation*}
q=2\Delta f-|\nabla f|^{2}+R=\frac{|\nabla u|^{2}}{u^{2}}-2\frac{\Delta u}{u}+R
\end{equation*}
is bounded for $t\in[0, T]$. If we choose $\epsilon$ sufficiently small, then $q(x,T)\leq\frac{2n}{T+\epsilon}+\frac{2n}{\epsilon}$. According to the maximum principle (\cite[Proposition 4.3]{Bennett Chow 2004}, c.f.), for all $t\in[0,T]$,
\begin{equation*}
q(x,t)\leq\frac{2n}{t+\epsilon}+\frac{2n}{\tau+\epsilon}.
\end{equation*}
Letting $\epsilon\rightarrow0$, we have for all $t\in[0, T]$,
\begin{equation*}
q(x,t)\leq\frac{2n}{t}+\frac{2n}{\tau}.
\end{equation*}
That is to say
\begin{equation*}
2\Delta f-|\nabla f|^{2}+R\leq\frac{2n}{t}+\frac{2n}{\tau}.
\end{equation*}
Further, using $f=-\ln u-\frac{n}{2}\ln(4\pi\tau)$, the above yields
\begin{equation*}
\frac{|\nabla u|^{2}}{u^{2}}-2\frac{u_{\tau}}{u}-R\leq\frac{2n}{t}+\frac{2n}{\tau}, \quad t\in[0, T].
\end{equation*}
This completes the proof of the theorem.
\end{proof}

\subsection{Additional gradient estimates}

In this subsection we present two more Hamilton type log gradient estimates for the forward  conjugate heat equation and the heat equation coupled with the Ricci flow. The main point is that the condition is only applied on the scalar curvature.

In 2006, one of us \cite[Theorem 3.3]{Zhang 2006} generalized Theorem \ref{Theorem A} to evolving metrics using Hamilton's \cite{Hamilton 1993} idea, without any curvature assumptions. In other words, he proved that when the metric $g(t)$ evolves by the Ricci flow (\ref{def. of Ricci flow}), positive solutions $u$ of the heat equation \eqref{def. of heat eq.} also satisfy the same estimate. This result was proved independently by X. D. Cao and R. S. Hamilton \cite[Theorem 5.1]{Cao Xiaodong and Hamilton}. Here, we consider the same problem for the forward conjugate heat equation \eqref{forward conjugate eq} under the Ricci flow (\ref{def. of Ricci flow}) and obtain the following Hamilton-type bound.

\begin{theorem}\label{Hamilton Thm of grad(u)}
Let $(\mathrm{\bf{M}}^{n},g(t)), t\in(0,1]$, be a solution to the Ricci flow (\ref{def. of Ricci flow}) on an n-dimensional compact Riemannian manifold whose scalar curvature satisfies $|R|\leq R_{0}$ for some positive constant $R_{0}$. Suppose $u: \mathrm{\bf{M}}\times(0,1]\rightarrow(0,\infty)$ is any positive solution to the forward conjugate heat equation \eqref{forward conjugate eq} on $\mathrm{\bf{M}}\times(0,1]$ and $A=\mathrm{sup}_{\mathrm{\bf{M}}\times(0,1]}u$, i.e.
\[
(\Delta_{g(t)}+R_{g(t)}-\partial_{t})u=0.
\]
Then there exists a large constant $\beta_{1}$, depending only on $R_{0}$ and $n$, such that
$$\frac{|\nabla u(x,t)|^{2}}{u^{2}(x,t)}\leq\frac{\beta_{1}}{t}\ln\left(\frac{2A}{u}\right)$$
for any $(x,t)\in \mathrm{\bf{M}}\times(0,1]$.
\end{theorem}

\begin{remark}
 We point out that the auxiliary term $\frac{u}{R+R_{0}}$ used in the following proof originated from \cite{Zhang and Zhu Meng} by M. Zhu and one of us.
\end{remark}

\begin{proof}[\bf{Proof}]
By direct calculation, we have
\begin{flalign*}
&(\Delta+R-\partial_{t})\left(\frac{|\nabla u|^{2}}{u}\right)
=\frac{2}{u}\left(u_{ij}-\frac{u_{i}u_{j}}{u}\right)^{2}-2\langle\nabla u,\nabla R\rangle,  \\
&(\Delta+R-\partial_{t})(Ru)=-2u|\text{Ric}|^{2}+2\langle\nabla u,\nabla R\rangle,  \\
&(\Delta+R-\partial_{t})\left(\frac{\beta_{2}u}{R+R_{0}}\right)
=\frac{2\beta_{2}u}{(R+R_{0})^{2}}|\text{Ric}|^{2}+\frac{2\beta_{2}u}{(R+R_{0})^{3}}|\nabla R|^{2}
-\frac{2\beta_{2}}{(R+R_{0})^{2}}\langle\nabla R,\nabla u\rangle, \\
&(\Delta+R-\partial_{t})(u\ln u)=\frac{|\nabla u|^{2}}{u}-Ru.
\end{flalign*}
Denote
$$\mathcal{Q}:=t\left(\frac{|\nabla u|^{2}}{u}-Ru+\frac{\beta_{2}u}{R+R_{0}}\right)
-\beta_{1}u\ln\left(\dfrac{2A}{u}\right).$$
We conclude that
\begin{equation*}
\begin{aligned}
&(\Delta+R-\partial_{t})\mathcal{Q}\\
=&\frac{2t}{u}\left(u_{ij}-\frac{u_{i}u_{j}}{u}\right)^{2}+2tu\left(1+\frac{\beta_{2}}{(R+R_{0})^{2}}\right)|\text{Ric}|^{2}
-2t\left(2+\frac{\beta_{2}}{(R+R_{0})^{2}}\right)\langle\nabla R,\nabla u\rangle\\
&+\frac{2\beta_{2}t}{(R+R_{0})^{3}}u|\nabla R|^{2}
+(\beta_{1}-1)\left(\frac{|\nabla u|^{2}}{u}-Ru\right)-\frac{\beta_{2}u}{R+R_{0}}.
\end{aligned}
\end{equation*}
It follows from Cauchy-Schwarz inequality that
\begin{equation*}
-2t\left(2+\frac{\beta_{2}}{(R+R_{0})^{2}}\right)\langle\nabla R,\nabla u\rangle
\geq-t[2(R+R_{0})^{2}+\beta_{2}]\left(\frac{u|\nabla R|^{2}}{(R+R_{0})^{3}}+\frac{|\nabla u|^{2}}{(R+R_{0})u}\right).
\end{equation*}
Substituting this to the previous identity, we deduce
\begin{equation*}
\begin{aligned}
&(\Delta+R-\partial_{t})\mathcal{Q}\\
\geq&\frac{2t}{u}\left(u_{ij}-\frac{u_{i}u_{j}}{u}\right)^{2}+2tu\left(1+\frac{\beta_{2}}{(R+R_{0})^{2}}\right)|\text{Ric}|^{2}
+\left(\frac{\beta_{2}}{(R+R_{0})^{3}}-\frac{2(R+R_{0})^{2}}{(R+R_{0})^{3}}\right)tu|\nabla R|^{2}\\
&+\left(\beta_{1}-1-\frac{t[2(R+R_{0})^{2}+\beta_{2}]}{R+R_{0}}\right)\left(\frac{|\nabla u|^{2}}{u}-Ru\right)
-\frac{t[2(R+R_{0})^{2}+\beta_{2}]}{R+R_{0}}Ru-\frac{\beta_{2}u}{R+R_{0}}.
\end{aligned}
\end{equation*}
Using the boundedness of the scalar curvature, we obtain
$$\frac{\beta_{2}}{(R+R_{0})^{3}}-\frac{2(R+R_{0})^{2}}{(R+R_{0})^{3}}\geq0$$
for some large constant $\beta_{2}$ depending only on $R_0$.

If $\mathcal{Q}$ has no zero point, then we have $$\mathcal{Q} \leq 0$$ since $\mathcal{Q}(0)<0$. The result follows by choosing an absolute positive constant $\beta_{1}$.

We assume that $(x_{0},t_{0})$ is the zero point of $\mathcal{Q}$, where the following holds
\begin{equation}\label{Q=0}
\begin{aligned}
t\left(\frac{|\nabla u|^{2}}{u}-Ru\right)
&=\beta_{1}u\ln\left(\frac{2A}{u}\right)-\frac{t\beta_{2}u}{R+R_{0}}\\
&\geq\left(\beta_{1}\ln2-\frac{t\beta_{2}}{R+R_{0}}\right)u\\
&\geq\frac{\beta_{1}\ln2}{2}u\\
&>0
\end{aligned}
\end{equation}
for a sufficiently large constant $\beta_{1}$.
Similarly, take a suitably large $\beta_{1}$ to make
$$\beta_{1}-1-\frac{t[2(R+R_{0})^{2}+\beta_{2}]}{R+R_{0}}\geq\frac{\beta_{1}}{2}.$$
Consequently
\begin{equation}\label{large beta1}
\begin{aligned}
\left(\beta_{1}-1-\frac{t[2(R+R_{0})^{2}+\beta_{2}]}{R+R_{0}}\right)\left(\frac{|\nabla u|^{2}}{u}-Ru\right)
\geq\frac{\beta_{1}}{2}\left(\frac{|\nabla u|^{2}}{u}-Ru\right).
\end{aligned}
\end{equation}
It then follows from (\ref{Q=0}) and (\ref{large beta1}) that
\begin{equation*}
\begin{aligned}
&\left(\beta_{1}-1-\frac{t[2(R+R_{0})^{2}+\beta_{2}]}{R+R_{0}}\right)\left(\frac{|\nabla u|^{2}}{u}-Ru\right)
-\frac{t[2(R+R_{0})^{2}+\beta_{2}]}{R+R_{0}}Ru-\frac{\beta_{2}u}{R+R_{0}}\\
&\geq\left(\frac{\beta_{1}}{2}\cdot\frac{\beta_{1}\ln2}{2t}-\frac{t[2(R+R_{0})^{2}+\beta_{2}]}{R+R_{0}}R-\frac{\beta_{2}}{R+R_{0}}\right)u\\
&\geq0
\end{aligned}
\end{equation*}
for some large enough $\beta_{1}$ depending only on $R_0$.

The maximum principle (\cite[Proposition 4.3]{Bennett Chow 2004}, c.f.) implies that $$\mathcal{Q}(t)\leq0$$ for all $t\in(0,1]$. That is to say
$$\frac{|\nabla u|^{2}}{u}-Ru
+\frac{\beta_{2}u}{R+R_{0}}\leq\frac{\beta_{1}}{t}u\ln\left(\frac{2A}{u}\right).$$
Hence
$$\frac{|\nabla u|^{2}}{u^{2}}\leq\frac{\beta_{1}}{t}\ln\left(\frac{2A}{u}\right), \quad t \in (0,1]$$
for an appropriate constant $\beta_{1}$ depending only on $R_0$.

In the above proof, to avoid $R+R_0=0$ in the denominator, we can replace $R_0$ by $R_0+\epsilon$ and let $\epsilon \to 0$ eventually.

The proof of the theorem is completed.
\end{proof}

As shown in Theorem \ref{Theorem B} above, Hamilton obtained a bound on $\Delta u$ under the condition that the Ricci curvature is bounded from below. When the metric evolves by Ricci flow (\ref{def. of Ricci flow}), we find that this condition can be replaced by that the scalar curvature being bounded. Details are presented in the following Theorem.

 To state the theorem, recall the quantity $\nu[g,\tau]:=\inf_{0<\tau'<\tau}\mu[g,\tau']$, where $\mu[g,\tau]$ denotes Perelman's $\mu$-entropy with parameter $\tau$.  See \cite{Perelman} or \cite{Zhang and Bamler} for more details e.g. As shown in the proof of Theorem \ref{Thm-reverse Harnack inequality}, the following  theorem can infer the backward Harnack inequality, but we will not carry it out this time.

\begin{theorem}\label{Hamilton Thm of Laplace(u)}
Let $(\mathrm{\bf{M}}^{n},g(t))$, $t\in(0,1)$, be a solution to the Ricci flow (\ref{def. of Ricci flow}) on an n-dimensional compact Riemannian manifold whose scalar curvature satisfies $|R|\leq R_{0}$ for some positive constant $R_{0}$. Suppose $u: \mathrm{\bf{M}}\times(0,1)\rightarrow(0,\infty)$ is a positive solution to the heat equation (\ref{def. of heat eq.}) in $\mathrm{\bf{M}}\times(0,1)$ and
$A=\sup_{\mathrm{\bf{M}}\times(0,1)} u(x, t)$. Then the following holds
$$\frac{\partial_{t}u}{u}\leq \frac{\zeta}{t}\ln\left(\frac{eA}{u}\right)$$
for a constant $\zeta$ depending only on $\nu[g_{0},2]$, $R_{0}$, $n$ and diameters  of $(\mathrm{\bf{M}}^{n},g(t))$.
\end{theorem}

To facilitate the proof of Theorem \ref{Hamilton Thm of Laplace(u)}, we first present the following Lemmas.

\begin{lemma}\label{evolution of Q}
Let $(\mathrm{\bf{M}}^{n},g(t))$ be a solution to the Ricci flow (\ref{def. of Ricci flow}) on an n-dimensional compact Riemannian manifold. Suppose $u$ is a solution associated with the heat equation (\ref{def. of heat eq.}) in $\mathrm{\bf{M}}\times(0,T)$. Then the function
$$Q=\Delta u+\delta\frac{|\nabla u|^{2}}{u}-\frac{\lambda_1}{t}(R+R_{0})u+\frac{\lambda_{2}}{t}\frac{u}{R+R_{0}}$$
satisfies the equality
\begin{equation*}
\begin{aligned}
(\Delta-\partial_{t})Q
=&\frac{1}{u}\left(u_{ij}-\frac{u_{i}u_{j}}{u}-uR_{ij}\right)^{2}
+\frac{2\delta-1}{u}\left(u_{ij}-\frac{u_{i}u_{j}}{u}\right)^{2}\\
&+\frac{1}{(\frac{2\lambda_1}{t}-1)u}\left(\left(\frac{2\lambda_1}{t}-1\right)uR_{ij}
-\frac{u_{i}u_{j}}{u}\right)^{2}-\frac{1}{\frac{2\lambda_1}{t}-1}\frac{|\nabla u|^{4}}{u^{3}}\\
&+\frac{\lambda_{1}u}{t(R+R_{0})}\left|\nabla R-\frac{R+R_{0}}{u}\nabla u\right|^{2}
-\frac{\lambda_{1}}{t^{2}}(R+R_{0})u-\frac{\lambda_{1}(R+R_{0})}{t}\frac{|\nabla u|^{2}}{u}\\
&+\frac{\lambda_{2}}{t^{2}}\frac{u}{R+R_{0}}+\left(\frac{\lambda_{2}}{t(R+R_{0})^{3}}
-\frac{\lambda_{1}}{t(R+R_{0})}\right)u|\nabla R|^{2}
+\frac{2\lambda_{2}u|\mathrm{Ric}|^{2}}{t(R+R_{0})^{2}}\\
&+\frac{\lambda_{2}u}{t(R+R_{0})}\left|\frac{\nabla R}{R+R_{0}}-\frac{\nabla u}{u}\right|^{2}
-\frac{\lambda_{2}}{t(R+R_{0})}\frac{|\nabla u|^{2}}{u},
\end{aligned}
\end{equation*}
where $\delta>\frac{1}{2}$, $\lambda_{1}>0$, $\lambda_{2}>0$ are constants.
\end{lemma}

\begin{proof}[\bf{Proof}]
We compute the heat operator acting on each term of $Q$ separately. First, a standard computation yields
$$(\Delta-\partial_{t})(\Delta u)=-2R_{ij}u_{ij}.$$
Also, we have
$$(\Delta-\partial_{t})\left(\frac{|\nabla u|^{2}}{u}\right)=\frac{2}{u}\left(u_{ij}-\frac{u_{i}u_{j}}{u}\right)^{2}.$$
Second, we obtain
$$(\Delta-\partial_{t})\left(\frac{1}{t}(R+R_{0})u\right)
=\frac{1}{t}\left(-2u|\mathrm{Ric}|^{2}+2\langle \nabla u, \nabla R \rangle\right)+\frac{1}{t^{2}}(R+R_{0})u.$$
Furthermore, we get
$$(\Delta-\partial_{t})\left(\frac{u}{t(R+R_{0})}\right)
=\frac{1}{t}\left(\frac{2u|\mathrm{Ric}|^{2}}{(R+R_{0})^{2}}+\frac{2u|\nabla R|^{2}}{(R+R_{0})^{3}}
-2\frac{\langle \nabla u, \nabla R\rangle}{(R+R_{0})^{2}}\right)+\frac{u}{t^{2}(R+R_{0})}.$$
Combining the above formulas and using the definition of $Q$ infers
\begin{equation*}
\begin{aligned}
(\Delta-\partial_{t})Q
=&-2R_{ij}u_{ij}+\frac{2\delta}{u}\left(u_{ij}-\frac{u_{i}u_{j}}{u}\right)^{2}
+\frac{2\lambda_{1}}{t}u|\mathrm{Ric}|^{2}
-\frac{2\lambda_{1}}{t}\langle \nabla u, \nabla R \rangle -\frac{\lambda_{1}}{t^{2}}(R+R_{0})u\\
&+\frac{2\lambda_{2}u|\mathrm{Ric}|^{2}}{t(R+R_{0})^{2}}
+\frac{2\lambda_{2}u|\nabla R|^{2}}{t(R+R_{0})^{3}}
-\frac{2\lambda_{2}\langle \nabla u, \nabla R\rangle}{t(R+R_{0})^{2}}
+\frac{\lambda_{2}u}{t^{2}(R+R_{0})}.
\end{aligned}
\end{equation*}
Thus, it follows from perfect square formula that
\begin{equation*}
\begin{aligned}
(\Delta-\partial_{t})Q
=&\frac{1}{u}\left(u_{ij}-\frac{u_{i}u_{j}}{u}-uR_{ij}\right)^{2}
+\frac{2\delta-1}{u}\left(u_{ij}-\frac{u_{i}u_{j}}{u}\right)^{2}
+\left(\frac{2\lambda_{1}}{t}-1\right)u|\mathrm{Ric}|^{2}\\
&-2R_{ij}\frac{u_{i}u_{j}}{u}
+\frac{1}{\frac{2\lambda_{1}}{t}-1}\frac{|\nabla u|^{4}}{u^{3}}
-\frac{1}{\frac{2\lambda_{1}}{t}-1}\frac{|\nabla u|^{4}}{u^{3}}
-\frac{2\lambda_{1}}{t}\langle \nabla u, \nabla R \rangle \\
&-\frac{\lambda_{1}}{t^{2}}(R+R_{0})u
+\frac{\lambda_{1}u|\nabla R|^{2}}{t(R+R_{0})}
+\frac{\lambda_{1}(R+R_{0})|\nabla u|^{2}}{tu}-\frac{\lambda_{1}(R+R_{0})|\nabla u|^{2}}{tu}\\
&+\left(\frac{\lambda_{2}}{t(R+R_{0})^{3}}-\frac{\lambda_{1}}{t(R+R_{0})}\right)u|\nabla R|^{2}
+\frac{2\lambda_{2}u|\mathrm{Ric}|^{2}}{t(R+R_{0})^{2}}
+\frac{\lambda_{2}u|\nabla R|^{2}}{t(R+R_{0})^{3}}\\
&-\frac{2\lambda_{2}\langle \nabla u, \nabla R\rangle}{t(R+R_{0})^{2}}
+\frac{\lambda_{2}|\nabla u|^{2}}{tu(R+R_{0})}
-\frac{\lambda_{2}|\nabla u|^{2}}{tu(R+R_{0})}+\frac{\lambda_{2}u}{t^{2}(R+R_{0})}.
\end{aligned}
\end{equation*}
Observe that the 3rd, 4th and 5th terms, 7th, 9th and 10th terms and 14th, 15th and 16th terms form complete squares, respectively. Hence we get Lemma \ref{evolution of Q}.
\end{proof}

\begin{lemma}\label{heat kernel lemma}
Let $(\mathrm{\bf{M}}^{n},g(t))$, $t\in(\frac{T}{2},T)$, be a solution to the Ricci flow (\ref{def. of Ricci flow}) on an n-dimensional compact Riemannian manifold, where $T\in(0,1)$ is a constant. Assume the scalar curvature is bounded, i.e., $|R|\leq R_{0}$ for some positive constant $R_{0}$. Suppose $u=u(x,t)=G(x,t;y,\frac{T}{4})$ is the heat kernel associated with the heat equation (\ref{def. of heat eq.}) in $\mathrm{\bf{M}}\times(\frac{T}{2},T)$. Then there exists a large constant $\rho$, only depending on $\nu[g_{0},2]$, $R_{0}$, diameters of $(\mathrm{\bf{M}}^{n},g(t))$ and $n$, such that
$$\frac{\partial_{t}u}{u}\leq \frac{\rho(1+t)}{t-\frac{T}{2}}\ln\left(\frac{eA}{u}\right)$$
for $A=\mathrm{sup}_{\mathrm{\bf{M}}\times(\frac{T}{2},T)}u$ and $(x,t)\in \mathrm{\bf{M}}\times(\frac{T}{2},T)$.
\end{lemma}

\begin{proof}[\bf{Proof}]
Pick $\rho \geq 1$ to be determined later, and denote
$$\Omega:=\left(t-\frac{T}{2}\right)Q-\rho(1+t)u\ln\left(\frac{eA}{u}\right).$$
The goal is to show $\Omega \leq 0$. By the maximum principle we only need to prove that $$(\Delta-\partial_{t}) \Omega \geq 0$$ where $\Omega \geq 0$. So the following calculation takes places where $\Omega \geq 0$.
It follows from Lemma \ref{evolution of Q} that
\begin{equation}\label{Q-ineq.}
\begin{aligned}
(\Delta-\partial_{t})\Omega
\geq&\left(t-\frac{T}{2}\right)\left\{\frac{1}{nu}\left(\Delta u-\frac{|\nabla u|^{2}}{u}-uR\right)^{2}
+\frac{2\delta-1}{nu}\left(\Delta u-\frac{|\nabla u|^{2}}{u}\right)^{2}\right.\\
&\left.+\frac{1}{n(\frac{2\lambda_{1}}{t}-1)u}\left[\left(\frac{2\lambda_{1}}{t}-1\right)uR
-\frac{|\nabla u|^{2}}{u}\right]^{2}
-\frac{1}{\frac{2\lambda_{1}}{t}-1}\frac{|\nabla u|^{4}}{u^{3}}\right\}\\
&+\left[\rho(1+t)-\frac{t-\frac{T}{2}}{t}\left(\frac{\lambda_{2}}{R+R_{0}}+\lambda_{1}(R+R_{0})\right)
-\delta\right]\frac{|\nabla u|^{2}}{u}-\Delta u\\
&-\frac{\lambda_{2}u}{t(R+R_{0})}+\rho u\ln\left(\frac{eA}{u}\right)
\end{aligned}
\end{equation}
for large constant $\lambda_{2}$ depending only on $R_0$ and $\lambda_1$ and $(x,t)\in \mathrm{\bf{M}}\times\left(\frac{T}{2},T\right)$, where we have thrown out some positive terms.

According to Theorem 1.4 of \cite{Zhang and Bamler}, we have
\begin{equation*}
u(x,t) \leq c_{1}T^{-\frac{n}{2}}\exp\left(-\frac{c_{2}\underline{D}^{2}}{t-\frac{T}{4}}\right),
\end{equation*}
where $c_{1}$ and $c_{2}$ only depend on $\nu[g_{0},2]$, $R_{0}$ and $n$, and $\underline{D}:=\inf_{t\in(0,1)}\mathrm{diam}_{t}$, here $\mathrm{diam}_{t}$ is the diameter induced by the metric $g(t)$. This implies that
\begin{equation*}
\ln\left(\frac{eA}{u(x,t)}\right) \geq \frac{c_{3}\underline{D}^{2}}{t-\frac{T}{4}},
\end{equation*}
where $c_{3}$ only depending on $\nu[g_{0},2]$, $R_{0}$ and $n$. Then by the boundedness of the scalar curvature $R$ and $t-\frac{T}{4}<t$, we derive
\begin{equation}\label{333}
\begin{aligned}
\rho u\ln\left(\frac{eA}{u}\right)-\frac{\lambda_{2}u}{t(R+R_{0})}
\geq& \rho u \frac{c_{3}\underline{D}^{2}}{t}-\frac{\lambda_{2}u}{t(R+R_{0})}\\
\geq& 0
\end{aligned}
\end{equation}
for a large constant $\rho$  only depending on $\nu[g_{0},2]$, $R_{0}$, $\underline{D}$ and $n$.
Consequently, \eqref{Q-ineq.} becomes
\begin{equation}\label{L-ineq.}
\begin{aligned}
(\Delta-\partial_{t})\Omega
\geq&\left(t-\frac{T}{2}\right)\left\{\frac{1}{nu}\left(\Delta u-\frac{|\nabla u|^{2}}{u}-uR\right)^{2}
+\frac{2\delta-1}{nu}\left(\Delta u-\frac{|\nabla u|^{2}}{u}\right)^{2}\right.\\
&\left.+\frac{1}{n(\frac{2\lambda_{1}}{t}-1)u}\left[\left(\frac{2\lambda_{1}}{t}-1\right)uR
-\frac{|\nabla u|^{2}}{u}\right]^{2}
-\frac{1}{\frac{2\lambda_{1}}{t}-1}\frac{|\nabla u|^{4}}{u^{3}}\right\}\\
&+\left[\rho(1+t)-\frac{t-\frac{T}{2}}{t}\left(\frac{\lambda_{2}}{R+R_{0}}+\lambda_{1}(R+R_{0})\right)
-\delta\right]\frac{|\nabla u|^{2}}{u}-\Delta u.
\end{aligned}
\end{equation}
Now we claim that
$$(\Delta-\partial_{t})\Omega \geq 0$$
and we see this by examining the following three cases.

Case 1: If
$$\Delta u\geq \psi\frac{|\nabla u|^{2}}{u},$$
where $\psi$ is a large positive constant, we have
$$\Delta u-\frac{|\nabla u|^{2}}{u}\geq(\psi-1)\frac{|\nabla u|^{2}}{u}.$$
This shows that
$$\frac{\delta-1}{nu}\left(\Delta u-\frac{|\nabla u|^{2}}{u}\right)^{2}
\geq \frac{\delta-1}{n}(\psi-1)^{2}\frac{|\nabla u|^{4}}{u^{3}}.$$
So we can cancel out the bad term $-\dfrac{1}{\frac{2\lambda_{1}}{t}-1}\dfrac{|\nabla u|^{4}}{u^{3}}$ by choosing the sufficiently large absolute constant $\psi$.
Moreover, according to $\frac{|\nabla u|^{2}}{u^{2}}\leq \frac{1}{t-\frac{T}{2}}\ln\left(\frac{A}{u}\right)$ (c.f. \cite[Theorem 3.3]{Zhang 2006}), we obtain
$$\delta\left(t-\frac{T}{2}\right)\frac{|\nabla u|^{2}}{u} \leq \delta u \ln\left(\frac{eA}{u}\right).$$
When $\Omega \geq 0$, i.e.,
$$\left(t-\frac{T}{2}\right)\Delta u
\geq \rho (1+t)u\ln\left(\frac{eA}{u}\right)-\delta\left(t-\frac{T}{2}\right)\frac{|\nabla u|^{2}}{u}
+\frac{t-\frac{T}{2}}{t}u\left[\lambda_{1}(R+R_{0})-\frac{\lambda_{2}}{R+R_{0}}\right],$$
we derive
$$\left(t-\frac{T}{2}\right)\Delta u
\geq [\rho (1+t)-\delta]u\ln\left(\frac{eA}{u}\right)
-\frac{t-\frac{T}{2}}{t} \frac{\lambda_{2}u}{R+R_{0}}.$$
Similar to \eqref{333}, we find
$$\rho t u\ln\left(\frac{eA}{u}\right)-\frac{t-\frac{T}{2}}{t} \frac{\lambda_{2}u}{R+R_{0}} \geq 0$$
for a large constant $\rho$ only depending on $\nu[g_{0},2]$, $R_{0}$, $\underline{D}$ and $n$. This indicates that
\begin{equation}\label{4444}
\left(t-\frac{T}{2}\right)\Delta u \geq (\rho-\delta)u.
\end{equation}
Using $\Delta u\geq \psi\frac{|\nabla u|^{2}}{u}$, we have
\begin{equation*}
\frac{\delta}{nu}\left(\Delta u-\frac{|\nabla u|^{2}}{u}\right)^{2} \geq \frac{\delta}{n}\left(1-\frac{1}{\psi}\right)^{2}\frac{(\Delta u)^{2}}{u}.
\end{equation*}
It follows from \eqref{4444} that
\begin{equation*}
\left(t-\frac{T}{2}\right)\frac{\delta}{nu}\left(\Delta u-\frac{|\nabla u|^{2}}{u}\right)^{2}-\Delta u
\geq \frac{\delta}{n}\left(1-\frac{1}{\psi}\right)^{2}(\rho-\delta)\Delta u-\Delta u
\geq 0
\end{equation*}
for any constant $\rho \geq \delta + \frac{n}{\delta(1-\frac{1}{\psi})^{2}}$.
Hence, by \eqref{L-ineq.}, we conclude that
$$(\Delta-\partial_{t})\Omega \geq 0$$
when $\rho$ is suitable.

Case 2: If
$$\Delta u\leq \frac{1}{2}\frac{|\nabla u|^{2}}{u},$$
we get
\begin{equation*}
\begin{aligned}
\left(\Delta u-\frac{|\nabla u|^{2}}{u}\right)^{2}
&=\left(\frac{|\nabla u|^{2}}{u}-\Delta u\right)^{2}\\
&\geq\frac{1}{4}\frac{|\nabla u|^{4}}{u^{2}}.
\end{aligned}
\end{equation*}
From this, we deduce that
$$\frac{2\delta-1}{nu}\left(\Delta u-\frac{|\nabla u|^{2}}{u}\right)^{2}\geq\frac{2\delta-1}{4n}\frac{|\nabla u|^{4}}{u^{3}},$$
having cancelled the bad term $-\dfrac{1}{\frac{2\lambda_{1}}{t}-1}\dfrac{|\nabla u|^{4}}{u^{3}}$ by taking $\delta$ sufficiently large.
So it follows from $\Delta u\leq \frac{1}{2}\frac{|\nabla u|^{2}}{u}$ that
$$(\Delta-\partial_{t})\Omega \geq 0$$
for sufficiently large $\rho$.

Case 3: If
$$\frac{1}{2}\frac{|\nabla u|^{2}}{u}\leq \Delta u \leq \psi\frac{|\nabla u|^{2}}{u},$$
then the squares terms in \eqref{L-ineq.} become useless, which are thrown away, obtaining
\begin{equation}\label{Case 3-L ineq.}
\begin{aligned}
(\Delta-\partial_{t})\Omega
\geq&\left(t-\frac{T}{2}\right)
\left[-\frac{1}{\frac{2\lambda_{1}}{t}-1}\frac{|\nabla u|^{4}}{u^{3}}\right]\\
&+\left[\rho(1+t)-\frac{t-\frac{T}{2}}{t}\left(\frac{\lambda_{2}}{R+R_{0}}+\lambda_{1}(R+R_{0})\right)
-\delta\right]\frac{|\nabla u|^{2}}{u}-\Delta u.
\end{aligned}
\end{equation}
It follows from Theorem 3.3 of \cite{Zhang 2006} that
$$\frac{|\nabla u|^{2}}{u^{2}}\leq \frac{1}{t-\frac{T}{2}}\ln\left(\frac{A}{u}\right).$$
This implies that
\begin{equation}\label{111}
\begin{aligned}
\frac{1}{\frac{2\lambda_{1}}{t}-1}\frac{|\nabla u|^{4}}{u^{3}}
=&\frac{t}{2\lambda_{1}-t}\frac{|\nabla u|^{2}}{u}\cdot\frac{|\nabla u|^{2}}{u^{2}}\\
\leq& \frac{t}{2\lambda_{1}-t}\cdot\frac{1}{t-\frac{T}{2}}\ln\left(\frac{A}{u}\right)
\frac{|\nabla u|^{2}}{u}.
\end{aligned}
\end{equation}
According to (2.11) of \cite{Zhang and Bamler},
we derive
$$u(x,t) \geq c_{4}T^{-\frac{n}{2}}\exp\left\{-\frac{c_{5}\overline{D}^{2}}{T}\right\},$$
where $c_{4}$ and $c_{5}$ only depend on $\nu[g_{0},2]$, $R_{0}$ and $n$, and $\overline{D}:=\sup_{t\in(0,1)}\mathrm{diam}_{t}$, here $\mathrm{diam}_{t}$ is the diameter induced by the metric $g(t)$.
Combining this with an upper bound on $u$ in \cite[Theorem 1.4]{Zhang and Bamler}, we obtain
$$\ln\left(\frac{A}{u}\right) \leq \frac{c_{6}\overline{D}^{2}}{T},$$
where $c_6$ only depends on $\nu[g_{0},2]$, $R_{0}$ and $n$.
Substituting the above inequality into \eqref{111}, we now obtain
\begin{equation*}
\begin{aligned}
\frac{1}{\frac{2\lambda_{1}}{t}-1}\frac{|\nabla u|^{4}}{u^{3}}
\leq \frac{c_{6}\overline{D}^{2}}{2\lambda_{1}-t}\cdot\frac{1}{t-\frac{T}{2}}\frac{|\nabla u|^{2}}{u}.
\end{aligned}
\end{equation*}
This shows that
\begin{equation*}
\begin{aligned}
\left(t-\frac{T}{2}\right)\left[-\frac{1}{\frac{2\lambda_{1}}{t}-1}\frac{|\nabla u|^{4}}{u^{3}}\right]
\geq-c_{6}\overline{D}^{2}\frac{|\nabla u|^{2}}{u}
\end{aligned}
\end{equation*}
for large $\lambda_{1}>1$.
Therefore, by the above inequality and $\Delta u \leq \psi\frac{|\nabla u|^{2}}{u}$, \eqref{Case 3-L ineq.} becomes
\begin{equation*}
\begin{aligned}
(\Delta-\partial_{t})\Omega
\geq \left[\rho(1+t)-c_{6}\overline{D}^{2}-(\psi+\delta)
-\frac{t-\frac{T}{2}}{t}\left(\frac{\lambda_{2}}{R+R_{0}}
+\lambda_{1}(R+R_{0})\right)\right]\frac{|\nabla u|^{2}}{u}.
\end{aligned}
\end{equation*}
This asserts that
$$(\Delta-\partial_{t})\Omega \geq 0$$
when $\rho$ is large enough depending only on $\nu[g_{0},2]$, $R_{0}$, $\overline{D}$ and $n$.

Since $\left(t-\frac{T}{2}\right)Q-\rho(1+t)u\ln\left(\frac{eA}{u}\right)\leq 0$ at $t=\frac{T}{2}$, the maximum principle (\cite[Proposition 4.3]{Bennett Chow 2004}, c.f.) implies
$$\left(t-\frac{T}{2}\right)Q-\rho(1+t)u\ln\left(\frac{eA}{u}\right)\leq 0$$
for all $\frac{T}{2} \leq t \leq T$, thus the result follows.

\end{proof}

Using the Lemma \ref{evolution of Q} and Lemma \ref{heat kernel lemma}, we now provide the proof of Theorem \ref{Hamilton Thm of Laplace(u)}.

\begin{proof}[\bf{Proof of Theorem \ref{Hamilton Thm of Laplace(u)}}]
 Given any $T\in(0,1)$, the heat kernel $G=G(x,t;y,\frac{T}{4})$ of the heat equation (\ref{def. of heat eq.}), it follows from the Lemma \ref{heat kernel lemma} that there exists a positive constant $\rho$, such that
\begin{equation}\label{Lemma 1}
\begin{aligned}
\frac{\partial_{t}G}{G}\leq\frac{2\rho}{t-\frac{T}{2}}\ln\left(\frac{eA_T}{G}\right)
\end{aligned}
\end{equation}
for $(x,t)\in \mathrm{\bf{M}}\times(\frac{T}{2},T)$, where $A_T=\text{sup}_{\mathrm{\bf{M}}\times(\frac{T}{2},T)}G(x,t;y,\frac{T}{4})$.

Using the fact that
$$u(x,t)=\int_{\mathrm{\bf{M}}}G(x,t;y,T/4)u(y,T/4)\text{d}y,$$
where $\text{d}y=\text{d}g(\frac{T}{4})(y)$ and $t\in(0,1)$, we deduce that
\begin{equation*}
\begin{aligned}
u\ln\left(\frac{eA}{u}\right)-\frac{t}{\zeta}\partial_{t}u
\ge &\int_{\mathrm{\bf{M}}}G(x,t;y,T/4)u(y,T/4)\text{d}y\cdot
\ln\left(\frac{e A_T \int_{\mathrm{\bf{M}}}u(y,\frac{T}{4})\text{d}y}
{\int_{\mathrm{\bf{M}}}G(x,t;y,\frac{T}{4})u(y,\frac{T}{4})\text{d}y}\right)\\
&-\frac{t}{\zeta}\int_{\mathrm{\bf{M}}}\partial_{t}G(x,t;y,T/4)u(y,T/4)\text{d}y.
\end{aligned}
\end{equation*}
Application of the integral form of Jensen's inequality to the first term on the right-hand side yields
\begin{equation*}
\begin{aligned}
\int_{\mathrm{\bf{M}}}&G(x,t;y,T/4)u(y,T/4)\text{d}y\cdot
\ln\left(\frac{e A_T \int_{\mathrm{\bf{M}}}u(y,\frac{T}{4})\text{d}y}
{\int_{\mathrm{\bf{M}}}G(x,t;y,\frac{T}{4})u(y,\frac{T}{4})\text{d}y}\right)\\
& = \int_{\mathrm{\bf{M}}}G(x,t;y,T/4)u(y,T/4)\text{d}y\cdot
\ln\left(\frac{ \int_{\mathrm{\bf{M}}}G(x, t; y, \frac{T}{4}) u(y,\frac{T}{4})
\cdot \frac{e A_T}{G(x, t;y, \frac{T}{4})} \text{d}y}
{\int_{\mathrm{\bf{M}}}G(x,t;y,T/4)u(y,T/4)\text{d}y}\right)\\
&\geq\int_{\mathrm{\bf{M}}}G(x,t;y,T/4)u(y,T/4)
\cdot\ln\left(\frac{e A_T}{G(x,t;y,\frac{T}{4})}\right)\text{d}y.
\end{aligned}
\end{equation*}
Now we choose $t\in[\frac{3}{4}T,T]$.
Then combining this with (\ref{Lemma 1}) and $t-\frac{T}{2} \geq \frac{T}{4} \geq \frac{t}{4}$, we obtain that
\begin{equation}\label{Jensen's result}
\begin{aligned}
u\ln\left(\frac{eA}{u}\right)-\frac{t}{\zeta}\partial_{t}u
&\ge \int_{\mathrm{\bf{M}}}G(x,t;y,T/4)u(y,T/4)
\cdot\ln\left(\frac{e A_T}{G(x,t;y,\frac{T}{4})}\right)\text{d}y\\
&-\frac{t}{\zeta}\int_{\mathrm{\bf{M}}}\frac{8\rho}{t}\ln\left(\frac{e A_T}{G(x,t;y,T/4)}\right)
\cdot G(x,t;y,T/4)u(y,T/4)\text{d}y\\
&\ge 0,
\end{aligned}
\end{equation}
for any $\zeta\geq 8\rho$ and $x\in \mathrm{\bf{M}}$ and $t\in [\frac{3}{4}T,T]$.
Since $T\in(0,1)$ is arbitrary, we deduce that \eqref{Jensen's result} holds for all $t\in(0,1)$.
Hence
$$u\ln\left(\frac{eA}{u}\right)-\frac{t}{\zeta}\partial_{t}u\geq0$$
for a large constant $\zeta\geq 8\rho$ depending only on $\nu[g_{0},2]$, $R_{0}$, $n$ and diameters.

\end{proof}

\section{Matrix estimates for positive solutions to the heat equation under the Ricci flow}\label{section 4}
In this section, we prove Theorem \ref{matric heat eq.-Thm} by using the idea of Hamilton's tensor maximum principle.

\begin{proof}[\bf{Proof of Theorem \ref{matric heat eq.-Thm}}]

Setting $\varepsilon = 1$ and $\delta = 0$ in (2.2) of \cite{Zhang Qi and Li Xiaolong}, we get that $H_{ij} := \nabla_i \nabla_j \ln u$ satisfies $$(\partial_t - \Delta) H_{ij} = 2 H_{ij}^{2} + 2 R_{ikjl} H_{kl} - R_{ik} H_{jk} - R_{jk} H_{ik} + 2 R_{ikjl} \nabla_k v \nabla_l v + 2 \nabla_k H_{ij} \nabla_k v,$$
where $v=\ln u$. For a positive constant $\xi$, define the auxiliary tensor  $$Z_{ij} := H_{ij} + \frac{\xi}{t} g_{ij}.$$
Direct calculations show that
\begin{equation}\label{evolution eq. of tilde Z}
\begin{aligned}
(\partial_t - \Delta) Z_{ij}
=& 2 Z_{ij}^2 - \frac{4\xi}{t} Z_{ij} + 2 R_{ikjl} Z_{kl} - R_{ik} Z_{jk}- R_{jk} Z_{ik}+ 2 R_{ikjl} \nabla_k v \nabla_l v \\
&+ 2 \nabla_k Z_{ij} \nabla_k v + \left(\frac{2\xi^2}{t^2} - \frac{\xi}{t^2}\right) g_{ij} - \frac{2\xi}{t} R_{ij}.
\end{aligned}
\end{equation}
 By the type III assumption
 $$|\mathrm{Ric}|_{g(t)} \leq  \frac{\sigma}{t}, $$ we have
\begin{equation*}
\begin{aligned}
\left(\frac{2\xi^2}{t^2} - \frac{\xi}{t^2}\right) g_{ij} - \frac{2\xi}{t} R_{ij}
&\geq \left(\frac{2\xi^2}{t^2} - \frac{\xi}{t^2} - \frac{2\xi \sigma}{t^2} \right) g_{ij} \\
&\geq 0
\end{aligned}
\end{equation*}
for  $\xi \geq \frac{1+2\sigma}{2}$.
Hence substituting this into (\ref{evolution eq. of tilde Z}) yields
$$(\partial_t - \Delta) Z_{ij} \geq F_{ij} + 2 \nabla_k Z_{ij} \nabla_k v,$$
where
$$F_{ij} = 2 Z_{ij}^2 - \frac{4\xi}{t} Z_{ij} + 2 R_{ikjl} Z_{kl}- R_{ik} Z_{jk} - R_{jk} Z_{ik}.$$
Given any $\varepsilon>0$, define $$Z_{\varepsilon}(t):=Z(t)+\varepsilon(1+t)g(t_0), $$
and we have $\nabla_{g(t)}Z_{\varepsilon}=\nabla_{g(t)}Z$, $\Delta_{g(t)}(Z_{\varepsilon})_{ij}=\Delta_{g(t)}Z_{ij}$ at $t_{0}$, where $Z(t)$ is a tensor represented by $Z_{ij}$.
It follows from $Z_{ij} \geq 0$ as $t \to 0^+$ that $(Z_{\varepsilon})_{ij} > 0$ as $t \to 0^+$. Next we show that $$(Z_{\varepsilon})_{ij}(x,t)>0, \quad\forall (x,t)\in \mathrm{\bf{M}} \times (0,T).$$
Suppose that $(x_{0}, t_{0})$ is a point where there exists a vector $\vartheta$ such that $(Z_{\varepsilon})(\vartheta,\vartheta)(x_{0},t_{0})=0$
for the first time. Hence
$$(Z_{\varepsilon})(V,V)(x,t)>0$$
for all vector $V$, $x\in\mathrm{\bf{M}}$, and $t<t_{0}$.
Using parallel transport, we extend $\vartheta$ along geodesic rays starting from $x_0$ so that it becomes a parallel vector field in a neighborhood of $x_0$ with respect to $g(t_0)$. This vector field, still denoted by $\vartheta = \vartheta(x)$, is regarded as stationary about the time.  We then have at $(x_{0},t_{0})$
\begin{equation}\label{Zij}
\begin{aligned}
0 &\geq \partial_{t}(Z_{\varepsilon}(\vartheta,\vartheta))
=\vartheta^i \vartheta^j \partial_{t}(Z_\varepsilon)_{ij}\\
&\geq \vartheta^i \vartheta^j (\Delta Z_{ij}+F_{ij} + 2 \nabla_k Z_{ij} \nabla_k v)
+ \varepsilon \vartheta^i \vartheta^j g_{ij}\\
&= \Delta (\vartheta^i \vartheta^j (Z_{\varepsilon})_{ij})
+2 \nabla_k (\vartheta^i \vartheta^j (Z_{\varepsilon})_{ij})+\tilde{F}_{ij}\vartheta^i \vartheta^j
+ \varepsilon \vartheta^i \vartheta^j g_{ij},
\end{aligned}
\end{equation}
where we use the fact that $Z_{\varepsilon}(\vartheta,\cdot)=0$ to obtain
\begin{equation*}
\begin{aligned}
\tilde{F}_{ij}\vartheta^i \vartheta^j=2\varepsilon^{2}(1+t)^{2}g_{ij}\vartheta^i \vartheta^j
+\frac{4\xi}{t}\varepsilon(1+t)g_{ij}\vartheta^i \vartheta^j
+2R_{ikjl}(Z_{\varepsilon})_{kl}\vartheta^i \vartheta^j.
\end{aligned}
\end{equation*}
Since $t_{0}$ is the first time and $Z_{\varepsilon}(0)>0$, we know $Z_{\varepsilon}$ is positive-semidefinite for any $t\leq t_{0}$. Choosing, at the point $x_0$, an orthonormal basis consisting of eigenvectors of $Z_{\varepsilon}$, we then have
\[
(Z_{\varepsilon})_{kl}=
\begin{cases}
  \nonumber \lambda_{k}, & \, \text{if} \,\quad  k=l, \\
  \nonumber 0,           & \, \text{if} \,\quad  k \neq l.
\end{cases}
\]
This shows that $\tilde{F}_{ij}\vartheta^i \vartheta^j \geq 0$, here we have used the non-negativity of the sectional curvature. Details can be found in \cite[p16]{Zhang Qi and Li Xiaolong}.
Combining with $\tilde{F}_{ij}\vartheta^i \vartheta^j \geq 0$ and \eqref{Zij}, we obtain
$$0 \geq \partial_{t}(Z_{\varepsilon}(\vartheta,\vartheta))>0$$ at $(x_{0},t_{0})$,
which is a contradiction. Hence $Z_{\varepsilon}(V,V)>0$ on $\mathrm{\bf{M}} \times (0, T)$ for all $\varepsilon > 0$ and by taking the limit as $\varepsilon \rightarrow 0$, we get $Z(V,V) \geq 0$ on $\mathrm{\bf{M}} \times (0, T)$. This proves the theorem.

\end{proof}


\begin{appendix}
\renewcommand{\thesection}{\Alph{section}}\section{}\label{appendix A}

In this section, we follow the idea of the proof of Lemma 4.1 in \cite{Hamilton 1993} (i.e., Theorem \ref{Theorem B} in this paper) to derive the specific value of the constant $\theta_{1}$ used in Theorem \ref{Thm-reverse Harnack inequality}, which also streamlines the original proof.

\begin{proof}[\bf{Proof of Theorem \ref{Theorem B}}]
By direct computation, we obtain
\begin{equation}\label{evolve of du/u-fixed metric}
(\partial_{t}-\Delta) \left( \frac{|\nabla u|^2}{u} \right)
= - \frac{2}{u} \left| \nabla_i \nabla_j u - \frac{\nabla_i u \nabla_j u}{u} \right|^2
-\frac{2}{u}\mathrm{Ric}(\nabla u, \nabla u).
\end{equation}
According to Cauchy-Schwarz inequality, we deduce
$$(\partial_{t}-\Delta) \left( \frac{|\nabla u|^2}{u} \right)
\leq - \frac{2}{nu} \left| \Delta u - \frac{|\nabla u|^2}{u} \right|^2 + 2 K \frac{|\nabla u|^2}{u}.$$
We denote
$$\varphi=\frac{e^{Kt}-1}{Ke^{Kt}},\quad
h=\varphi \left( \Delta u + \frac{|\nabla u|^2}{u} \right)
- u \left( n + 4 \ln \left( \frac{A}{u} \right) \right).$$
Then using $\partial_{t}\varphi=e^{-Kt}$ and $\partial_{t}(\Delta u)-\Delta(\partial_{t}u)=0$, we conclude
\begin{equation*}
(\partial_{t}-\Delta)h
=e^{-Kt} \left( \Delta u + \frac{|\nabla u|^2}{u} \right)
+ \varphi \left(\partial_{t} - \Delta \right) \left( \frac{|\nabla u|^2}{u} \right)
- 4 \frac{|\nabla u|^2}{u}.
\end{equation*}
Thus, it follows from $\varphi'+K\varphi=1$ and \eqref{evolve of du/u-fixed metric} that
\begin{equation*}
(\partial_{t}-\Delta)h
\leq e^{-Kt} \left( \Delta u - \frac{|\nabla u|^2}{u} \right)
-\frac{2 \varphi}{nu} \left| \Delta u - \frac{|\nabla u|^2}{u} \right|^2
- 2 \frac{|\nabla u|^2}{u}.
\end{equation*}
Now we claim that
$$(\partial_{t}-\Delta)h \leq 0$$
whenever $h \geq 0$ and we see this by examining the following three cases.
\begin{subequations}
\begin{align*}
     &\text{(i) if $\Delta u \leq \frac{|\nabla u|^2}{u}$ we are done, since $\partial_{t}\varphi=e^{-Kt} \geq 0$ .}\\
     &\text{(ii) if $\frac{|\nabla u|^2}{u} \leq \Delta u \leq 3\frac{|\nabla u|^2}{u}$ we are also done, since $\partial_{t}\varphi=e^{-Kt} \leq 1$.}\\
     &\text{(iii) if $3\frac{|\nabla u|^2}{u} \leq \Delta u$ and $h>0$, then}\\
     &\makebox[\linewidth]{$\displaystyle 2\left( \Delta u - \frac{|\nabla u|^2}{u} \right) \geq \Delta u + \frac{|\nabla u|^2}{u}\geq \frac{nu}{\varphi} $} \\
     &\text{and since $\partial_{t}\varphi=e^{-Kt} \leq 1$ we are done completely.}
\end{align*}
\end{subequations}
Note that $h \leq 0$ at $t = 0$, the maximum principle implies $h \leq 0$ for all $t$.
Since $e^{Kt} - 1 \geq Kt$, it follows from the definition of $\varphi$ that
$t \leq e^{Kt}\varphi(t) \leq e^{K}\varphi(t)$ for $0 \leq t \leq 1$.
This indicates that
\begin{equation*}
t\Delta u
\leq t\left(\Delta u + \frac{|\nabla u|^2}{u}\right)
\leq e^{K}\varphi(t)\left(\Delta u + \frac{|\nabla u|^2}{u}\right).
\end{equation*}
Finally, using $h \leq 0$, we have
\begin{equation*}
t\Delta u
\leq e^{K}u\left(n+4\ln\left(\frac{A}{u}\right)\right)
\leq \theta_{1}\left(1+\ln\left(\frac{A}{u}\right)\right)u,
\end{equation*}
where $\theta_{1}=e^{K}\max\{n,4\}$.
\end{proof}

\renewcommand{\thesection}{ \Alph{section}}
\section{}\label{appendix B}
In this section, under the condition that the manifold has nonnegative Ricci curvature, we compute the constant $\theta_{2}$ in Theorem \ref{Thm-reverse Harnack inequality} using the upper and lower bounds of the heat kernel $G(x,t;y,0)$. Since the constant $\theta_{2}$ is determined by \eqref{Zhang-CAG-log(A/u) bound} in the proof of Theorem \ref{Thm-reverse Harnack inequality}, we only need to estimate $\ln\left(\frac{eA}{G(x,t;y,0)}\right)$. The details are as follows.

According to Theorem 13.4 and Theorem 13.8 in \cite{Geometric analysis}, the following upper and lower bounds for $G(x,t;y,0)$ on a manifold with nonnegative Ricci curvature hold
\[
2^{-\frac{n}{2}} e^{-2} (2\pi)^{-\frac{n}{2}} \frac{\alpha_{n-1}}{n} \frac{1}{|B(x,\sqrt{t})|} \exp\left(-\frac{d^2(x,y)}{4t}\right)
\leq G(x,t;y,0) \leq e^2 \frac{1}{|B(x,\sqrt{t})|} \exp\left(-\frac{d^2(x,y)}{12t}\right)
\]
for any $t \in [\frac{T}{2}, T]$, where $T \leq 1$ and $\alpha_{n-1}$ is the area of the unit $(n-1)$ - sphere.
These upper and lower bounds imply that
\[
A = \sup_{\mathrm{\bf{M}} \times [\frac{T}{2},T]} G(x,t;y,0) \leq \frac{e^2}{\inf_{z \in M} |B(z,\sqrt{\frac{T}{2}})|}
\]
and
\[
\frac{1}{G(x,t;y,0)}
\leq 2^{\frac{n}{2}} e^2 (2\pi)^{\frac{n}{2}} \frac{n}{\alpha_{n-1}} |B(x,\sqrt{T})| \exp\left(\frac{\text{diam}^2}{2T}\right),
\]
which yields that
\begin{equation}\label{upper bound of A/G}
\frac{A}{G(x,t;y,0)}
\leq 2^{\frac{n}{2}} e^4 (2\pi)^{\frac{n}{2}} \frac{n}{\alpha_{n-1}} \exp\left(\frac{\text{diam}^2}{2T}\right)
\frac{|B(x,\sqrt{T})|}{\inf_{z \in M} |B(z,\sqrt{\frac{T}{2}})|}.
\end{equation}
Assuming $\inf_{z \in M} |B(z,\sqrt{\frac{T}{2}})|=|B(z_{0},\sqrt{\frac{T}{2}})|$, we obtain the following estimates via the volume comparison theorem.

(a) If $\sqrt{T} > 2d(x,z_0)$, then
\[
|B(x,\sqrt{T})| \leq 2^n |B(z_0,\sqrt{T})|.
 \]

(b) If $\sqrt{T} \leq 2d(x,z_0)$, then
\[
|B(x,\sqrt{T})|
\leq 4^n |B(z_0,\sqrt{T})| \left(\frac{4d(x,z_0)+\sqrt{T}}{\sqrt{T}}\right)^n.
 \]
(The above estimates can be found in \cite[p177]{Peter Li and Yau}.)
Therefore, in any case we have
\[ |B(x,\sqrt{T})| \leq 4^n \left(1+\frac{4 \text{diam}}{\sqrt{T}}\right)^n |B(z_0,\sqrt{T})|. \]
Applying the volume comparison theorem once more together with the above estimate, we obtain
\[
\frac{|B(x,\sqrt{T})|}{\inf_{z \in M} |B(z,\sqrt{\frac{T}{2}})|}
= \frac{|B(x,\sqrt{T})|}{|B(z_0,\sqrt{\frac{T}{2}})|}
\leq 4^n \left(1+\frac{4 \text{diam}}{\sqrt{T}}\right)^n (\sqrt{2})^n.
 \]
Combining with \eqref{upper bound of A/G}, we have
\[
\frac{A}{G(x,t;y,0)}
\leq 2^{3n} e^4 (2\pi)^{\frac{n}{2}} \frac{n}{\alpha_{n-1}} \cdot
\left(1+\frac{4 \text{diam}}{\sqrt{T}}\right)^n \exp\left(\frac{\text{diam}^2}{2T}\right).
 \]
Taking the logarithm on both sides of the above inequality, we obtain
\[
\ln\left(\frac{A}{G(x,t;y,0)}\right)
\leq \ln\left(2^{3n} e^4 (2\pi)^{\frac{n}{2}} \frac{n}{\alpha_{n-1}}\right) + n + 8n^2 + \frac{\text{diam}^2}{T},
\]
where we used $\ln x \leq x$. Hence
\[
\ln\left(\frac{eA}{G(x,t;y,0)}\right)
\leq \left[1 + n+ 18n^2 + \ln\left(2^{3n} e^4 (2\pi)^{\frac{n}{2}} \frac{n}{\alpha_{n-1}}\right)\right] \left(1 + \frac{\text{diam}^2}{t}\right)
\]
since $t \leq T$.
\hfill\qedsymbol

This concludes the estimate of constant $\theta_{2}$ in Theorem \ref{Thm-reverse Harnack inequality} when $\mathrm{Ric} \geq 0$.

\end{appendix}

\section*{Acknowledgement}
The authors would like to thank Professor Yu Zheng and Ms. Xuenan Fu for their helpful comments.
The first named author gratefully acknowledges the support from the National Natural Science Foundation of China (No.~12271163), Key Laboratory of MEA (Ministry of Education), the Science and Technology Commission of Shanghai Municipality (No.~22DZ2229014), and Shanghai Key Laboratory of PMMP (No.~18DZ2271000).
The support of Simons Foundation grant 710364 is gratefully acknowledged by the third named author.


\end{document}